\documentclass[10pt; a4paper]{amsart} %\documentclass[Ñ¡Ïî±í]{Àà}[°æ±ŸºÅyyyy/mm/dd]
\usepackage[left=2.50cm, right=2.50cm, top=2.50cm, bottom=2.50cm]{geometry} 
\usepackage{amscd, amsfonts, amsmath, amssymb, amsthm, arydshln, fixmath, graphicx, mathrsfs, overpic, tikz, tikz-cd}
\usepackage{hyperref}
\usepackage[all]{xy} %\usepackage[Ñ¡Ïî±í]{ºê°ü}[°æ±ŸºÅyyyy/mm/dd] ÊýÑ§×ÖÌåºê°ü:mathrsfs, amsfonts, amssymb
\usepackage{tikz}
\usepackage{pgfplots}
\usepackage{pgflibraryarrows}
\usepackage{pgflibrarysnakes}
\usepackage{dsfont}
\makeindex

\theoremstyle{definition} %±êÌâÓë±àºÅÎªºÚÌå, ÕýÎÄÎªÕý³£×ÖÌå
\newtheorem{Unity}{Unity}[section] %\newtheorem{¶šÀí»·Ÿ³Ãû}{±êÌâ}[Ö÷ŒÆÊýÆ÷Ãû]
\newtheorem*{Definition*}{Definition} %\newtheorem*{¶šÀí»·Ÿ³Ãû}[ÒÑ¶šÒå¶šÀí»·Ÿ³Ãû]{±êÌâ} ÊÖ¶¯±àºÅ, ²»×Ô¶¯±àºÅ
\newtheorem{Definition}[Unity]{Definition} %\newtheorem{¶šÀí»·Ÿ³Ãû}[ÒÑ¶šÒå¶šÀí»·Ÿ³Ãû]{±êÌâ} Óëµ±Ç°»·Ÿ³¹²ÓÃÍ¬Ò»žöÐòºÅŒÆÊýÆ÷

\theoremstyle{plain} %±êÌâÓë±àºÅÎªºÚÌå, ÕýÎÄÎªÐ±Ìå
\newtheorem*{Theorem*}{Theorem}
\newtheorem{Theorem}[Unity]{Theorem}
\newtheorem{Proposition}[Unity]{Proposition}
\newtheorem{Corollary}[Unity]{Corollary}
\newtheorem{Lemma}[Unity]{Lemma}

\theoremstyle{remark} %±êÌâÓë±àºÅÎªÐ±Ìå, ÕýÎÄÎªÕý³£×ÖÌå
\newtheorem*{Remark*}{Remark}
\newtheorem{Remark}[Unity]{Remark}

\numberwithin{Unity}{section}%\numberwithin{ŒÆÊýÆ÷}{Ö÷ŒÆÊýÆ÷}

\newcommand{\Id}{\mathrm{Id}}

\newcommand{\Hom}{\mathrm{Hom}}

\newcommand{\Rep}{\mathrm{Rep}}

\newcommand{\Spec}{\mathrm{Spec\,}}

\newcommand{\Qcoh}{\mathfrak{Qcoh}}
\newcommand{\Vect}{\mathfrak{Vect}}

\newcommand{\coker}{\mathrm{coker\,}}

\begin{document}

\title{The K\"unneth Formula of Fundamental Group Schemes}
\author{Lingguang Li}
\address{School of Mathematical Sciences,
Key Laboratory of Intelligent Computing and Applications (Tongji University), Ministry of Education, Shanghai 200092, CHINA}
\email{LiLg@tongji.edu.cn}
\author{Niantao Tian}
\address{School of Mathematical Sciences,
Key Laboratory of Intelligent Computing and Applications (Tongji University), Ministry of Education, Shanghai 200092, CHINA}
\email{tianniantao@tongji.edu.cn}
\begin{abstract} Let $k$ be a field, $f:X\rightarrow S$ a proper morphism between connected schemes proper over $k$, $x\in X(k)$ lying over $s\in S(k)$, $X_s$ the fibre of $f$ over $s$, $\mathcal{C}_X$, $\mathcal{C}_{S}$, $\mathcal{C}_{X_s}$ Tannakian categories over $X,S,X_s$ respectively, $\pi(\mathcal{C}_X,x)$, $\pi(\mathcal{C}_S,s)$, $\pi(\mathcal{C}_{X_s},x)$ the Tannaka group schemes respectively. We give a unified criterion for the exactness of the homotopy sequence of Tannakian fundamental group schemes $\pi(\mathcal{C}_{X_s},x)\rightarrow \pi(\mathcal{C}_X,x)\rightarrow \pi(\mathcal{C}_S,s)\rightarrow 1$. In particular, we obtain the equivalent conditions for the K\"unneth formula of fundamental group schemes for the product $X\times_k Y$ of two connected schemes $X$ and $Y$ proper over $k$. As an application, we obtain the K\"unneth formula of certain fundamental group schemes over any field, such as S, N, EN, F, EF, \'et, E\'et, Loc, ELoc and uni-fundamental group schemes.
\end{abstract}
\maketitle
\tableofcontents

\section{Introduction}

Grothendieck\cite{Gro60} introduced the étale fundamental group. Let $k$ be an algebraically closed field, $f:X\rightarrow S$ a separable proper surjective morphism with geometrically connected fibres between locally Noetherian connected schemes, $\bar{x}\rightarrow X$ a geometric point lying over a geometric point $\bar{s}\rightarrow S$, and $X_{\bar{s}}$ the fibre at $\bar{s}$. Grothendieck\cite{Gro61} proved that there is an exact homotopy sequence for the \'{e}tale fundamental group:
$$\pi^{\acute{e}t}_1(X_s,\bar{x})\rightarrow\pi^{\acute{e}t}_1(X,\bar{x})\rightarrow\pi^{\acute{e}t}_1(S,\bar{s})\rightarrow 1.$$
In particular, let $Y$ be a locally Noetherian connected scheme over $k$, $\bar{y}\rightarrow Y$ a geometric point. Grothendieck\cite{Gro61} proved that the Künneth formula holds for \'{e}tale fundamental group:
$$\pi^{\acute{e}t}_1(X\times_k Y,(x,y))\xrightarrow{\cong} \pi^{\acute{e}t}_1(X,x)\times_k \pi^{\acute{e}t}_1(Y,y).$$

Nori\cite{Nor82} introduced the Nori fundamental group scheme. Let $k$ be a field, $X$ a reduced connected locally noetherian scheme over $k$, $x \in X(k)$. The Nori fundamental group scheme is defined as the projective limit 
$$\pi^N(X,x):=\varprojlim\limits_{N(X,x)} G,$$ 
where $N(X,x)$ consists of triples $(P,G,p)$ where $P$ is an FPQC $G$-torsor, $G$ is a finite group scheme and $p\in P(k)$ is a rational point lying over $x$. If $k$ is a field and $X$ is a geometrically reduced connected scheme proper over $k$, then the category of finite dimensional representations of $\pi^N(X,x)$ is equivalent to $\mathcal{C}^N(X)$ whose objects are essentially finite vector bundles. Let $k$ be a perfect field, $f:X\rightarrow S$ a separable proper surjective morphism with geometrically connected fibres between locally Noetherian reduced connected proper schemes over $k$, $x\in X(k)$ lying over $s\in S(k)$. Esnault, Hai \& Viehweg\cite{EHV09} and  Zhang\cite{Zha13} gave necessary and sufficient conditions for the exactness of the homotopy sequence of the Nori fundamental group scheme
$$\pi^N(X_s,x)\rightarrow \pi^N(X,x)\rightarrow\pi^N(S,s)\rightarrow 1.$$
A counterexample in the projective case is stated by Esnault, Hai \& Viehweg\cite{EHV09}. Let $k$ be an algebraically closed field, $X$ and $Y$ reduced connected schemes proper over $k$, $x\in X(k)$, $y\in Y(k)$. Mehta \& Subramanian\cite{MeSu02} showed that the Künneth formula holds for Nori fundamental group scheme:
$$\pi^N(X\times_k Y,(x,y))\xrightarrow{\cong}\pi^N(X,x)\times_k \pi^N(Y,y).$$

Nori\cite{Nor82} introduced the unipotent fundamental group scheme. Let $k$ be a field, $X$ a scheme of finite type over $k$ with $H^0(X,\mathcal{O}_X)=k$, $x\in X(k)$. The unipotent fundamental group scheme $\pi^{uni}(X,x)$ is the Tannaka group scheme of the Tannakian category $\mathcal{C}^{uni}(X)$ whose objects consist of unipotent vector bundles on $X$. Let $k$ be a field, $X$ and $Y$ schemes of finite type proper over $k$ with $H^0(X,\mathcal{O}_X)=H^0(Y,\mathcal{O}_Y)=k$, $x\in X(k)$, $y\in Y(k)$. Nori\cite{Nor82} showed that the Künneth formula holds for unipotent fundamental group scheme:
$$\pi^{uni}(X\times_k Y,(x,y))\xrightarrow{\cong}\pi^{uni}(X,x)\times_k \pi^{uni}(Y,y).$$

Mehta \& Subramanian\cite{MeSu08} introduced the local fundamental group scheme. Let $k$ be an algebraically closed field of characteristic $p>0$, $X$ a smooth projective variety over $k$, $x\in X(k)$. The local fundamental group scheme $\pi^{Loc}(X,x)$ is the Tannaka group scheme of the Tannakian category $\mathcal{C}^{Loc}(X)$ whose objects consist of Frobenius trivial bundles. Let $k$ be a perfect field, $X$ and $Y$ reduced connected schemes proper over $k$, $x\in X(k)$, $y\in Y(k)$. Zhang\cite{Zha13} showed the Künneth formula holds for local fundamental group scheme:
$$\pi^{Loc}(X\times_k Y,(x,y))\xrightarrow{\cong}\pi^{Loc}(X,x)\times_k \pi^{Loc}(Y,y).$$

Amrutiya \& Biswas\cite{AmBi10} introduced the F-fundamental group scheme. Let $k$ be an algebraically closed field of characteristic $p>0$, $X$ a smooth projective variety over $k$, $x\in X(k)$. The F-fundamental group scheme $\pi^F(X,x)$ is the Tannaka group scheme of the Tannakian category $\mathcal{C}^F(X)$ which is the Tannakian subcategory of $\mathcal{C}^{NF}(X)$ generated by all Frobenius finite vector bundles on $X$. Let $k$ be an algebraically closed field, $X$ and $Y$ smooth projective varieties over $k$, $x\in X(k)$, $y\in Y(k)$. Amrutiya \& Biswas\cite{AmBi10} showed that the Künneth formula holds for F-fundamental group scheme:
$$\pi^{F}(X\times_k Y,(x,y))\xrightarrow{\cong}\pi^{F}(X,x)\times_k \pi^{F}(Y,y).$$

Langer\cite{Lan11} introduced the S-fundamental group scheme. Let $k$ be an algebraically closed field, $X$ a reduced connected scheme proper over $k$, $x\in X(k)$. The S-fundamental group scheme $\pi^{S}(X,x)$ is the Tannaka group scheme of Tannakian category $\mathcal{C}^{NF}(X)$ whose objects consist of numerically flat vector bundles on $X$. Let $k$ be an algebraically closed field, $X$ and $Y$ reduced connected schemes proper over $k$, $x\in X(k)$, $y\in Y(k)$. Langer\cite{Lan12} showed that the Künneth formula holds for S-fundamental group scheme:
$$\pi^{S}(X\times_k Y,(x,y))\xrightarrow{\cong}\pi^{S}(X,x)\times_k \pi^{S}(Y,y).$$

Let $k$ be a field, $X$ a smooth algebraic $k$-variety, $x\in X(k)$. Then the stratified fundamental group scheme $\pi^{str}(X,x)$ is defined to be the Tannaka group scheme of $\mathcal{D}_X$ whose objects consist of $D_X$-modules over $X$. Let $k$ be a field of characteristic 0, $f:X\rightarrow S$ a smooth proper morphism with geometrically connected fibres between two connected schemes of finite type over $k$, $x\in X(k)$ lying over $s\in S(k)$. Zhang\cite{Zha14} proved the exactness of homotopy sequence of stratified fundamental group scheme:
$$\pi^{str}(X_s,x)\rightarrow \pi^{str}(X,x)\rightarrow\pi^{str}(S,s)\rightarrow 1.$$
Let $k$ be an algebraically closed field, $f:X\rightarrow S$ be a smooth projective morphism with geometrically connected fibres between two smooth connected $k$-schemes, $x\in X(k)$ lying over $s\in S(k)$. Dos Santos\cite{dS15} proved that the exactness of homotopy sequence of stratified fundamental group scheme:
$$\pi^{str}(X_s,x)\rightarrow \pi^{str}(X,x)\rightarrow\pi^{str}(S,s)\rightarrow 1.$$ 
Consequently, the Künneth formula also holds for the stratified fundamental group scheme.

Otabe\cite{Ota17} introduced the EN-fundamental group scheme. Let $k$ be a field of characteristic zero, $X$ a reduced connected scheme proper over $k$, $x\in X(k)$. The EN-fundamental group scheme $\pi^{EN}(X,x)$ is the Tannaka group scheme of the Tannakian category $C^{EN}(X)$ whose objects consist of semi-finite vector bundles on $X$. Let $k$ be a field of characteristic zero, $X$ and $Y$ reduced connected schemes proper over $k$, $x\in X(k)$, $y\in Y(k)$. Otabe\cite{Ota17} showed that the Künneth formula holds for EN-fundamental group scheme:
$$\pi^{EN}(X\times_k Y,(x,y))\xrightarrow{\cong}\pi^{EN}(X,x)\times_k \pi^{EN}(Y,y).$$

$\mathbf{Notations:}$ Let $k$ be a field, $X$ a scheme over $k$, $x \in X(k)$, $\mathcal{C}_X$ a Tannakian category whose objects consist of vector bundles on $X$ with fibre functor $|_x:E\mapsto E|_x$, $\pi(\mathcal{C}_X,x)$ the Tannaka group scheme of $\mathcal{C}_X$.

\begin{Theorem}[Theorem~\ref{Main}]
Let $k$ be a field, $f:X\rightarrow S$ a morphism of connected schemes proper over $k$ with $f_*\mathcal{O}_X\cong\mathcal{O}_S$, $x\in X(k)$ lying over $s\in S(k)$, $\mathcal{C}_X,\mathcal{C}_S,\mathcal{C}_{X_s}$ the Tannakian categories over $X,S,X_s$ respectively. If pullback induces functors $\mathcal{C}_S\xrightarrow{f^*}\mathcal{C}_X\xrightarrow{|_{X_s}}\mathcal{C}_{X_s}$. Then the following conditions are equivalent:
\begin{itemize}
    \item[(1)] \begin{itemize}
            \item[i.] The functor $|_{X_s}:\mathcal{C}_X\rightarrow \mathcal{C}_{X_s}$ is observable;
            \item[ii.] For any $E\in \mathcal{C}_X$, $f_* E\in\mathcal{C}_S$ and $E$ satisfies base change at $s$.
        \end{itemize} 
	\item[(2)]  \begin{itemize}
            \item[i.] The functor $|_{X_s}:\mathcal{C}_X\rightarrow \mathcal{C}_{X_s}$ is observable;
            \item[ii.] For any subset $M\subseteq\mathcal{C}_X$ 
             and its corresponding principal $\pi(\langle M\rangle_{\mathcal{C}_X},x)$-principal bundle $\phi_M: P_M\rightarrow X$, there exists a $($for any$)$ filtered colimit system $\{{\mathcal{E}_M}_{\alpha}\}_{\alpha\in I}$ of subsheaves of ${\phi_M}_*\mathcal{O}_{P_M}$, s.t. ${\phi_M}_*\mathcal{O}_{P_M}=\varinjlim{\mathcal{E}_M}_{\alpha}$ where ${\mathcal{E}_M}_{\alpha}\in \mathcal{C}_X$, $f_* {\mathcal{E}_M}_{\alpha}\in\mathcal{C}_S$ and ${\mathcal{E}_M}_{\alpha}$ satisfies base change at $s$.
                \end{itemize}
    \item[(3)] \begin{itemize}
            \item[i.] The functor $|_{X_s}:\mathcal{C}_X\rightarrow \mathcal{C}_{X_s}$ is observable;
            \item[ii.] Let $\phi: P\rightarrow X$ be the principal $\pi(\mathcal{C}_X,x)$-bundle corresponding to $\mathcal{C}_X$. There exists a $($for any$)$ filtered colimit system $\{\mathcal{E}_\alpha\}_{\alpha\in I}$ of subsheaves of $\phi_*\mathcal{O}_P$, s.t. $\phi_*\mathcal{O}_P=\varinjlim\mathcal{E}_{\alpha}$ where $\mathcal{E}_{\alpha}\in \mathcal{C}_X$, $f_* \mathcal{E}_{\alpha}\in\mathcal{C}_S$ and $\mathcal{E}_{\alpha}$ satisfies base change at $s$.
            \end{itemize}
\end{itemize}
The above equivalent conditions imply the following exact homotopy sequence $$\pi(\mathcal{C}_{X_s},x)\rightarrow\pi(\mathcal{C}_X,x)\rightarrow\pi(\mathcal{C}_S,s)\rightarrow 1$$
Moreover, if we suppose further that $S$ is integral and $f$ is a flat morphism, then the exactness of homotopy sequence also implies the above three equivalent conditions.
\end{Theorem}

\begin{Theorem}[Theorem~\ref{kunnethMain}]
    Let $k$ be a field, $X$ and $Y$ connected schemes proper over $k$, $x_0\in X(k)$, $y_0\in Y(k)$, $p_X: X\times_k Y\rightarrow X$ and $p_Y:X\times_k Y\rightarrow Y$ the projections, $i_X: X\rightarrow X\times_k Y$ send $x$ to $(x,y_0)$ and $i_Y: Y\rightarrow X\times_k Y$ send $y$ to $(x_0,y)$, $\mathcal{C}_X,\mathcal{C}_Y,\mathcal{C}_{X\times_k Y}$ the Tannakian categories over $X$, $Y$, $X\times_k Y$ respectively. Suppose pullback induce functors among Tannakian categories
$$\begin{aligned}
    \begin{tikzcd}
	Y\arrow[r,shift left=2pt,harpoon,"i_Y"]\arrow[d]&X\times_k Y\arrow[d,shift left=2pt,harpoon,"p_X"]\arrow[l,shift left=2pt,harpoon,"p_Y"]\\
	\Spec k\arrow[r]&X\arrow[u,shift left=2pt,harpoon,"i_X"]
    \end{tikzcd}&\quad\quad&\begin{tikzcd}
	\mathcal{C}_Y\arrow[r,shift left=2pt,harpoon,"p_Y^*"]&\mathcal{C}_{X\times_k Y}\arrow[d,shift left=2pt,harpoon,"i_X^*"]\arrow[l,shift left=2pt,harpoon,"i_Y^*"]\\
	&\mathcal{C}_X\arrow[u,shift left=2pt,harpoon,"p_X^*"]
    \end{tikzcd}\end{aligned}.$$
Then the following five conditions are equivalent:
\begin{enumerate}
    \item For any $E\in \mathcal{C}_{X\times_k Y}$, ${p_X}_*E\in\mathcal{C}_X$ and $E$ satisfies base change at $x_0$.
    \item For any $E\in \mathcal{C}_{X\times_k Y}$, ${p_Y}_*E\in\mathcal{C}_Y$ and $E$ satisfies base change at $y_0$.
    \item The homotopy sequence $1\rightarrow \pi(\mathcal{C}_Y,y_0)\rightarrow \pi(\mathcal{C}_{X\times_k Y},(x_0,y_0))\rightarrow \pi(\mathcal{C}_X,x_0)\rightarrow 1$ is exact.
    \item The homotopy sequence $1\rightarrow \pi(\mathcal{C}_X,x_0)\rightarrow \pi(\mathcal{C}_{X\times_k Y},(x_0,y_0))\rightarrow \pi(\mathcal{C}_Y,y_0)\rightarrow 1$ is exact.
    \item The natural homomorphism $\pi(\mathcal{C}_{X\times_k Y},(x_0,y_0))\rightarrow \pi(\mathcal{C}_X,x_0)\times_k \pi(\mathcal{C}_Y,y_0)$ is an isomorphism.
\end{enumerate}
\end{Theorem}

\begin{Proposition}[Proposition~\ref{kunnethdescent}]
    Let $k$ be a field, $X$ and $Y$ connected schemes proper over $k$, $x_0\in X(k)$, $y_0\in Y(k)$, $p_X: X\times_k Y\rightarrow X$ and $p_Y:X\times_k Y\rightarrow Y$ the projections, $i_X: X\rightarrow X\times_k Y$ send $x$ to $(x,y_0)$ and $i_Y: Y\rightarrow X\times_k Y$ send $y$ to $(x_0,y)$, $\mathcal{C}'_X,\mathcal{C}'_Y,\mathcal{C}'_{X\times_k Y}$ the Tannakian categories over $X$, $Y$, $X\times_k Y$ respectively, $\mathcal{C}_X\subseteq \mathcal{C}'_X,\mathcal{C}_Y\subseteq \mathcal{C}'_Y,\mathcal{C}_{X\times_k Y}\subseteq \mathcal{C}'_{X\times_k Y}$ the Tannakian subcategories. Suppose pullback induce functors among Tannakian categories
$$\begin{aligned}
    \begin{tikzcd}
	Y\arrow[r,shift left=2pt,harpoon,"i_Y"]\arrow[d]&X\times_k Y\arrow[d,shift left=2pt,harpoon,"p_X"]\arrow[l,shift left=2pt,harpoon,"p_Y"]\\
	\Spec k\arrow[r]&X\arrow[u,shift left=2pt,harpoon,"i_X"]
    \end{tikzcd}&\quad\quad&\begin{tikzcd}
	\mathcal{C}'_Y\arrow[r,shift left=2pt,harpoon,"p_Y^*"]&\mathcal{C}'_{X\times_k Y}\arrow[d,shift left=2pt,harpoon,"i_X^*"]\arrow[l,shift left=2pt,harpoon,"i_Y^*"]\\
	&\mathcal{C}'_X\arrow[u,shift left=2pt,harpoon,"p_X^*"]
    \end{tikzcd}&\quad\quad&\begin{tikzcd}
	\mathcal{C}_Y\arrow[r,shift left=2pt,harpoon,"p_Y^*"]&\mathcal{C}_{X\times_k Y}\arrow[d,shift left=2pt,harpoon,"i_X^*"]\arrow[l,shift left=2pt,harpoon,"i_Y^*"]\\
	&\mathcal{C}_X\arrow[u,shift left=2pt,harpoon,"p_X^*"]
    \end{tikzcd}\end{aligned}.$$
    If the natural homomorphism $\pi(\mathcal{C}'_{X\times_k Y},(x_0,y_0))\rightarrow \pi(\mathcal{C}'_X,x_0)\times_k \pi(\mathcal{C}'_Y,y_0)$ is an isomorphism, then the natural homomorphism $\pi(\mathcal{C}_{X\times_k Y},(x_0,y_0))\rightarrow \pi(\mathcal{C}_X,x_0)\times_k \pi(\mathcal{C}_Y,y_0)$ is an isomorphism.
\end{Proposition}
As an application, we obtain the K\"unneth formula of certain fundamental group schemes over any field, such as S, N, EN, F, EF, \'et, E\'et, Loc, ELoc and uni-fundamental group schemes.

\section{Preliminaries}

Let $k$ be a field, $K/k$ a field extension, $f:X\rightarrow S$ a morphism of schemes over $k$, $s\in S(k)$, $G$ an affine group scheme over $k$, $\Qcoh(X)$ the category of quasi-coherent sheaves on $X$, $\Vect(X)$ the category of vector bundles on $X$. Consider the following Cartesian diagrams
\[
    \begin{aligned}
        \begin{tikzcd}
     X_s\arrow{r}{t} \arrow{d}{g} & X \arrow{d}{f}\\
     \Spec(\kappa(s))\arrow{r}{s} & S
\end{tikzcd}&\quad\quad&\begin{tikzcd}
            X\times_{\Spec k}\Spec K\arrow[r,"p"]\arrow[d]& X\arrow[d]\\
            \Spec K\arrow[r]&\Spec k
        \end{tikzcd}&\quad\quad&\begin{tikzcd}
            G\times_{\Spec k}\Spec K\arrow[r,"p"]\arrow[d]& G\arrow[d]\\
            \Spec K\arrow[r]&\Spec k
        \end{tikzcd}
    \end{aligned}
\]
We denote $X_K := X \times_{\Spec k} \Spec K$, $G_K := G \times_{\Spec k}\Spec K$ and $x_K \in X_K(K)$ a point lying over $x\in X(k)$. For a coherent sheaf $E$ on $X$, we denote its pullback $p^*E$ simply by $E \otimes_k K$. We say a sheaf $F$ on $X$ \emph{satisfies base change at $s$} if the canonical map $s^* f_*F\rightarrow g_* t^* F$
is an isomorphism, i.e. $ f_*E|_s\xrightarrow{\cong} H^0(X_s,E|_{X_s})$.

\begin{Definition}
    Let $k$ be a field, $\mathcal{C},\mathcal{D}$ neutral Tannakian categories over $k$ with fibre functor $\omega_x$, $\Phi:\mathcal{C}\rightarrow \mathcal{D}$ an exact tensor functor. The functor $\Phi$ is said to be \textit{observable} if for any $E\in\mathcal{C}$ and any subobject $L\hookrightarrow \Phi(E)\in\mathcal{D}$ of rank 1, there exists $F\in\mathcal{C}$ and integer $n>0$ such that $(L^{\otimes n})^{\vee}\hookrightarrow \Phi(F)$ is a subobject in $\mathcal{D}$.
\end{Definition}

\begin{Lemma}\label{observable}
    Let $k$ be a field, $\mathcal{C},\mathcal{D}$ neutral Tannakian categories over $k$ with fibre functor $\omega_x$, $\Phi:\mathcal{C}\rightarrow \mathcal{D}$ an exact tensor functor. 
Then the following conditions are equivalent:
    \begin{enumerate}
        \item The functor $\Phi$ is observable.
        \item For any $E\in\mathcal{C}$ and any subobject $L\hookrightarrow \Phi(E)\in\mathcal{D}$ of rank 1, there exists $F\in\mathcal{C}$, s.t. $L^{\vee}\hookrightarrow \Phi(F)\in\mathcal{D}$.
        \item For any $E\in\mathcal{C}$ and any quotient $\Phi(E)\twoheadrightarrow E'\in\mathcal{D}$, there exists $F\in \mathcal{C}$, s.t. $E'\hookrightarrow \Phi(F)\in\mathcal{D}$.
        \item For any $E\in\mathcal{C}$ and any subobject $E'\hookrightarrow\Phi(E)\in\mathcal{D}$, there exists $F\in \mathcal{C}$, s.t. $\Phi(F)\twoheadrightarrow E' \in\mathcal{D}$.
        \item For any $E\in\mathcal{C}$ and any subobject $E'\hookrightarrow\Phi(E)\in\mathcal{D}$, there exists $F\in \mathcal{C}$, s.t. ${E'}^\vee\hookrightarrow \Phi(F)\in\mathcal{D}$.
    \end{enumerate}
\end{Lemma}

\begin{proof}
    $(1)\Leftrightarrow (2)$ By the isomorphism $L\otimes L^\vee\cong \mathds{1}$, it follows immediately.

    $(1)\Leftrightarrow (3)$ It follows by \cite[Proposition~A.3]{DaEs22}.

    $(3)\Leftrightarrow (4)$ Taking duality, it follows immediately.

    $(4)\Leftrightarrow (5)$ Taking duality, it follows immediately.
\end{proof}

\begin{Definition}
    Let $k$ be a field, $X$ a connected scheme proper over $k$, $x\in X(k)$, $\mathcal{C}_X$ a Tannakian category over $X$ whose objects consist of vector bundles on $X$ with fibre functor $|_x:E\mapsto E|_x$, and we denote its Tannaka group scheme by $\pi(\mathcal{C}_X,x)$.
\end{Definition}

\begin{Definition}
    Let $k$ be a field, $X$ a connected scheme proper over $k$, $x\in X(k)$, $\mathcal{C}_X$ the Tannakian category over $X$. Define the \textit{saturation category} $\overline{\mathcal{C}}_X$ of $\mathcal{C}_X$ as the full subcategory of $\Vect(X)$ whose objects are those $E$ for which there exists a filtration
    $$0\hookrightarrow E_1 \hookrightarrow \cdots\hookrightarrow E_n=E,$$
    such that $E^i=E_{i+1}/E_i\in\mathcal{C}_X$ for any $i$.
\end{Definition}

Let $k$ be a field, $X$ a connected scheme proper over $k$, $x\in X(k)$, $\mathcal{C}_X$ a Tannakian category whose objects consist of vector bundles on $X$ with fibre functor $|_x$, and $\pi(\mathcal{C}_X,x)$ its corresponding Tannaka group scheme. For any subset $M$ of $\mathcal{C}_X$, we denote $$M_1:=\{E~|~E\in M\text{ or }E^\vee\in M\}\text{, }M_2:=\{~E_1\otimes E_2\otimes\cdots \otimes E_m~|~E_i\in M_1, 1\leq i\leq m, m\in\mathbb{N}\},$$
$$\langle M\rangle_{\mathcal{C}_X}:=
\left\{ 
    E\in \mathcal{C}_X \left\vert 
    \begin{split}
        & \exists F_i\in M_2,1\leq i\leq t, \text{ and } E_1,E_2\in \mathcal{C}_X\\
        & s.t. \text{ }E_1\hookrightarrow E_2\hookrightarrow\bigoplus^{t}_{i=1}F_i,\text{ and } E\cong E_2/E_1
    \end{split}\right.
\right\}.
$$
Then $(\langle M\rangle_{\mathcal{C}_X},\otimes,|_x,\mathcal{O}_X)$ is a Tannakian subcategory of $(\mathcal{C}_X,\otimes,|_x,\mathcal{O}_X)$ and we denote its Tannaka group scheme by $\pi(\langle M\rangle_{\mathcal{C}_X},x)$. Then we have:
\begin{itemize}
      \item[(a)] For any subsets $N\subseteq M\subseteq \mathcal{C}_X$,
          there is a natural surjective homomorphism of affine group schemes $$\rho_{N}^M: \pi(\langle M\rangle_{\mathcal{C}_X},x)\twoheadrightarrow \pi(\langle N\rangle_{\mathcal{C}_X},x).$$
      \item[(b)] $\pi(\mathcal{C}_X,x)$ is the inverse limit of $\pi(\langle M\rangle_{\mathcal{C}_X},x)$, where $M$ runs through all subsets of $\mathcal{C}_X$.
\end{itemize}

Let $k$ be a field, $G$ an affine group scheme over $k$, $X$ a scheme proper over $k$. Denote by $\Rep_k^f(G)$ the finite dimensional representation category of $G$ and by $\Rep_k(G)$ the representation category of $G$. The regular representation of $G$ on $k[G]$ is given by $(g\cdot f)(h)=f(gh)$ for any $g,h\in G$ and $f\in k[G]$. Given a principal $G$-bundle $\phi:P\rightarrow X$, we can define an exact tensor functor 
$$\begin{aligned}
    \eta_X^P:\Rep_k^f(G)\rightarrow \Qcoh(X),&V\mapsto (\mathcal{O}_P\otimes_k V)^G, 
\end{aligned}$$
where the action is the diagonal action. Note that for any $V\in\Rep_k(G)$, $V$ is a union of its finite dimensional $G$-invariant subspaces. So for $V\in \Rep_k(G)$, we may define $$\eta_X^P(V):=\varinjlim\eta_X^P(V_\alpha),$$
where $V_\alpha$ takes over all finite dimensional $G$-invariant subspace of $V$. Conversely, given an exact tensor functor $\eta_X: \Rep_k(G)\rightarrow \Qcoh(X)$, one can define a principal $G$-bundle $\phi:P\rightarrow X$ by $P:=\Spec (\eta_X(k[G]))$. 

Consequently, there is a one-to-one correspondence between principal $G$-bundles $\phi:P\rightarrow X$ and exact tensor functors $\eta_X^P:\Rep_k(G)\rightarrow \Qcoh(X)$. Since $\eta^{\mathcal{C}_X}:\Rep_k^f(\pi(\mathcal{C}_X,x)\cong \mathcal{C}_X\hookrightarrow \Qcoh(X)$ is a fully faithful exact tensor functor, we obtain a principal $\pi(\mathcal{C}_X,x)$-bundle $\phi:P\rightarrow X$ such that $\eta^{\mathcal{C}_X}=\eta_X^P$. Moreover, 
\[\eta_X^P(k[\pi(\mathcal{C}_X,x)])=\phi_*\mathcal{O}_{P}=\varinjlim\mathcal{E}_{\alpha},\]
where $\mathcal{E}_\alpha=\eta_X^P(W_\alpha)$ running through all finite dimensional $\pi(\mathcal{C}_X,x)$-invariant subspaces $W_\alpha$ of $k[\pi(\mathcal{C}_X,x)]$.

\begin{Lemma}[{\cite[Chapter~VII, Proposition~2.3 \& 7.2]{Mil12}}]\label{isofield}
    Let $k$ be a field, $K/k$ a field extension, $f:G\rightarrow H$ a homomorphism of affine group schemes over $k$. Then $f$ is isomorphic iff $f_K:G_K\rightarrow H_K$ is isomorphic.
\end{Lemma}

\begin{Lemma}[{\cite[Chapter VIII, Proposition 10.2]{Mil12}}]\label{regularrepresentation}
    Let $k$ be a field, $G$ an affine group scheme over $k$, $(V, \rho)\in \Rep_k^f(G)$. Then $V$ embeds into a finite sum of copies of the regular representation, i.e. there exists integer $n>0$ and a $G$-invariant homomorphism $V\hookrightarrow k[G]^{\oplus n}$ where $k[G]$ is the regular representation.
\end{Lemma}

Let $k$ be a field, $G$ be an affine group scheme over $k$, $N$ a normal subgroup scheme of $G$, $\pi: G\rightarrow  G/N$ the quotient homomorphism. Then there is a natural functor 
$$\begin{aligned}
    \pi^*:\Rep_k(G/N)\rightarrow \Rep_k(G),(V,\rho)\mapsto(V,\rho\circ \pi),
\end{aligned}$$
For any $W\in\Rep_k(G)$, the subspace $W^N$ consisting of $N$-invariant vector of $W$ is a $G$-invariant subspace of $W$ since $N$ is a normal subgroup of $G$. Then we can define a functor 
$$\begin{aligned}
    -^N:\Rep_k(G)\rightarrow \Rep_k(G/N),(W,\rho)\mapsto(W^N,\bar{\rho} ),
\end{aligned}$$
where $G/N$ acts on $W^N$ by $\bar{\rho}(gN)\cdot w= \rho(g) \cdot w$ for $g\in G$ and $w\in W^N$.

\begin{Lemma}[{\cite[Lemma 6.4]{Jan03}}]\label{adjoint}
    Let $G$ be an affine group scheme over a field $k$, $N$ a normal subgroup scheme of $G$, $\pi: G\rightarrow  G/N$ the quotient homomorphism. Then the functor $-^N:\Rep_k(G)\rightarrow \Rep_k(G/N)$is right adjoint to $\pi^*:\Rep_k^f(G/N)\rightarrow \Rep_k^f(G)$.
\end{Lemma}

\begin{Theorem}[{\cite[Theorem A.1]{EHS07}}]\label{Thm1}
Let $L\xrightarrow{q}G\xrightarrow{p} A$ be a sequence of homomorphisms of affine group schemes over a field $k$. It induces a sequence of functors:
$$\Rep_k^f(A)\xrightarrow{p^*}\Rep_k^f(G)\xrightarrow{q^*}\Rep_k^f(L),$$
where $\Rep_k^f$ denotes the category of finite dimensional representations over $k$. Then we have the following:
\begin{enumerate}
    \item[(1)] The group homomorphism $p: G\rightarrow A$ is surjective $($faithfully flat$)$ iff  $p^*\Rep_k^f(A)$ is a full subcategory of $\Rep_k^f(G)$ and closed under taking subobjects $($subquotients$)$.
    \item[(2)] The group homomorphism $q:L\rightarrow G$ is injective $($a closed immersion$)$ iff any object of $\Rep_k^f(L)$ is a subquotient of an object of the form $q^*(V)$ for some $V\in \Rep_k^f(G)$.
    \item[(3)] If $p$ is faithfully flat. Then the sequence $L\xrightarrow{q}G\xrightarrow{p} A$ is exact iff the following conditions are fulfilled:
    \begin{enumerate}
        \item[(a)] For an object $V\in \Rep_k^f(G)$,  $q^* V\in\Rep_k^f(L)$ is trivial iff $V\cong p^*U$ for some $U\in\Rep_k^f(A)$.\label{3a}
        \item[(b)]\label{3b} Let $W_0$ be the maximal trivial subobject of $q^*V$ in $\Rep_k^f(L)$. Then there exists $V_0\hookrightarrow V$ in $\Rep_k^f(G)$ such that $q^*V_0\cong W_0$.
        \item[(c)] For any $W\in\Rep_k^f(G)$ and any quotient $q^*W\twoheadrightarrow W'\in\Rep_k^f(L)$, there exists $V\in \Rep_k^f(G)$ and an embedding $W'\hookrightarrow q^*V$.
    \end{enumerate}
\end{enumerate}
\end{Theorem}

\section{The Homotopy Sequence of Fundamental Group Schemes}

\textbf{Notation and Convention:}
Let $k$ be a field, $f:X\rightarrow S$ a morphism of connected schemes proper over $k$, $x\in X(k)$ lying over $s\in S(k)$, $\mathcal{C}_X,\mathcal{C}_S$ Tannakian categories over $X,S$ respectively, $\pi(\mathcal{C}_X,x),\pi(\mathcal{C}_S,s)$ the corresponding Tannaka group schemes. The pullback $f^*$ induces a functor $f^*:\mathcal{C}_S\rightarrow \mathcal{C}_X$ and a canonical homomorphism $\pi(\mathcal{C}_X,x)\rightarrow \pi(\mathcal{C}_S,s)$ of group schemes.

\begin{Lemma}\label{fullfaithful}
    Let $f:X\rightarrow S$ be a morphism of schemes with $f_*\mathcal{O}_X\cong \mathcal{O}_S$, then the induced functor $f^*:\Vect(S)\rightarrow \Vect(X)$ is fully faithful.
\end{Lemma}

\begin{proof}
    For any $E_1,E_2\in\Vect(S)$, by projection formula we have that 
$$\begin{aligned}
    f_*\mathcal{H}om_{\mathcal{O}_X}(f^*E_1,f^*E_2)&\cong f_*((f^*E_1)^{\vee}\otimes_{\mathcal{O}_X} f^*E_2)\\
    &\cong f_*f^*(E_1^{\vee}\otimes_{\mathcal{O}_S} E_2)\\
    &\cong f_*(\mathcal{O}_X\otimes_{\mathcal{O}_X}f^*(E_1^{\vee}\otimes_{\mathcal{O}_S}E_2))\\
    &\cong f_*\mathcal{O}_X\otimes_{\mathcal{O}_S}(E_1^\vee\otimes_{\mathcal{O}_S}E_2)\\
    & \cong \mathcal{O}_S\otimes_{\mathcal{O}_S}E_1^\vee\otimes_{\mathcal{O}_S} E_2\\
    &\cong \mathcal{H}om_{\mathcal{O}_S}(E_1,E_2).
\end{aligned}$$
Taking global section, we have $\Hom_{\mathcal{O}_X}(f^*E_1,f^*E_2)\cong \Hom_{\mathcal{O}_S}(E_1,E_2)$, i.e. $f^*$ is fully faithful.
\end{proof}

\begin{Proposition}\label{generalsur}
Let $k$ be a field, $f:X\rightarrow S$ a morphism of connected schemes proper over $k$ with $f_*\mathcal{O}_X\cong\mathcal{O}_S$, $x\in X(k)$ lying over $s\in S(k)$. Then the following conditions are equivalent:
\begin{enumerate}
    \item The induced homomorphism $\pi(\mathcal{C}_X,x)\rightarrow \pi(\mathcal{C}_S,s)$ is faithfully flat.
    \item For any $E\in\mathcal{C}_S$, and subobject $F\hookrightarrow f^*E\in\mathcal{C}_X$, we have $f_*F\in\mathcal{C}_S$ and $f^*f_*F\cong F$.
    \item For any $E\in \mathcal{C}_S$ and subobject $F\hookrightarrow f^*E\in\mathcal{C}_X$, then
    \begin{enumerate}
    \item $f_*F\in \mathcal{C}_S$.
    \item $F|_{X_s}$ $($resp. $F|_{X_{t}}$ for any $t\in S)$ is a trivial vector bundle .
    \item $F$ satisfies base change at $s$ $($resp. any $t\in S)$.
\end{enumerate}
\end{enumerate}
\end{Proposition}

\begin{proof}
$(1)\Rightarrow (2)$ If $\pi(\mathcal{C}_X,x)\rightarrow \pi(\mathcal{C}_S,s)$ is faithfully flat. Since $f^*$ is stable under taking subobjects by Theorem~\ref{Thm1}, for any $E\in\mathcal{C}_S$, and any subobjects $F\hookrightarrow f^*E\in\mathcal{C}_X$, there exists subobject $E'\hookrightarrow E\in\mathcal{C}_S$ such that $F\cong f^*E'$. By $f_*\mathcal{O}_X\cong \mathcal{O}_S$ and projection formula, we have 
$$f_*F\cong f_*f^*E'\cong f_*(\mathcal{O}_X\otimes_{\mathcal{O}_X} f^*E')\cong f_*\mathcal{O}_X\otimes_{\mathcal{O}_S} E'\cong E'\in\mathcal{C}_S.$$
Hence $f^*f_*F\cong f^*E'\cong F$.

$(2)\Rightarrow (3)$ (a) Obviously.

(b) For any $t\in S$, since $f^*f_*F\cong F$, we have $F|_{X_t}\equiv (f_*F)|_t\otimes_k \mathcal{O}_{X_t}$, i.e. $F|_{X_t}$ is a trivial vector bundle.

(c) For any $t\in S$, restricting the isomorphism $f^*f_*F\cong F$ to the fibre $X_t$, we have 
$$f^*f_*F|_{X_t}\equiv (f_*F)|_t\otimes_k \mathcal{O}_X\cong F|_{X_t}.$$
Taking global section, we have $(f_*F)|_t\cong H^0(X_t,F|_{X_t})$. In other words, $F$ satisfies base change at any $t\in S$.

$(3)\Rightarrow (1)$ Let $E\in \mathcal{C}_S$ and $F\hookrightarrow f^*E\in\mathcal{C}_X$ a subobject. On the one hand, since $F|_{X_s}$ is a trivial vector bundle, we have $H^0(X_s,F|_{X_s})\otimes_k \mathcal{O}_X\cong F|_{X_s}$. On the other hand, $F$ satisfies base change at $s$, so we have
$f^*f_*F|_{X_s}\equiv (f_*F)|_s\otimes_k \mathcal{O}_X\cong H^0(X_s,F|_{X_s})\otimes_k \mathcal{O}_X$. Therefore, we obtain
$$f^*f_*F|_{X_s}\cong F|_{X_s}.$$
Consider the exact sequence 
$$\ker\hookrightarrow f^*f_*F\rightarrow F\twoheadrightarrow \coker.$$
Restricting it to $X_s$, we obtain an exact sequence
$$\ker|_{X_s}\hookrightarrow f^*f_*F|_{X_s}\rightarrow F|_{X_s}\twoheadrightarrow \coker|_{X_s}.$$ Then $\ker|_{X_s}=\coker|_{X_s}=0$. Since $\mathcal{C}_X$ is an abelian category, we have $\ker,\coker\in\mathcal{C}_X$. It follows that $\ker=\coker=0$, i.e. $f^*f_*F\cong F$. Applying $f_*$ to $F\hookrightarrow f^*E$, we obtain $f_*F\hookrightarrow f_*f^*E\cong f_*\mathcal{O}_X\otimes_{\mathcal{O}_S}E\cong E$. Since $f_*F\in\mathcal{C}_S$, $f_*F\hookrightarrow E\in\mathcal{C}_S$ and $f^*f_*F\cong F\hookrightarrow f^*E\in\mathcal{C}_X$, we have $f^*$ is stable under taking subobjects. By Lemma~\ref{fullfaithful}, we have $f^*$ is fully faithful. Hence $\pi(\mathcal{C}_X,x)\rightarrow \pi(\mathcal{C}_S,s)$ is faithfully flat by Theorem~\ref{Thm1}.
\end{proof}

\begin{Proposition}\label{MaximalTrivial}
    Let $k$ be a field, $f:X\rightarrow S$ a morphism of connected schemes proper over $k$ with $f_*\mathcal{O}_X\cong\mathcal{O}_S$, $s\in S(k)$. Let $E\in\mathcal{C}_X$ satisfy the following conditions:
    \begin{enumerate}
        \item $f^*f_*E\in\mathcal{C}_X$.
        \item $E$ satisfies base change at $s$.
    \end{enumerate}
    Then the canonical map $f^*f_*E\rightarrow E$ is an embedding in $\mathcal{C}_X$.
\end{Proposition}

\begin{proof}
    Consider the exact sequence $0\rightarrow \ker \rightarrow f^*f_*E\rightarrow E$, where $\ker\in\mathcal{C}_X$ since $\mathcal{C}_X$ is abelian. Restricting it to $X_s$, since $E$ satisfies base change at $s$, we have
$$0\rightarrow \ker|_{X_s}\rightarrow(f^*f_* E)\vert_{X_s}\equiv f_* E\vert_s \otimes_k \mathcal{O}_{X_s}\cong H^0(X_s,E\vert_{X_s})\otimes_k\mathcal{O}_{X_s}\hookrightarrow E\vert_{X_s}.$$
Then we have $\ker|_{X_s}=0$, so $\ker=0$. Thus $f^*f_*E \hookrightarrow E$ is a subobject in $\mathcal{C}_X$.
\end{proof}    

\begin{Proposition}\label{embedding}
    Let $k$ be a field, $f:X\rightarrow S$ a morphism of schemes proper over $k$, $x\in X(k)$ lying over $s\in S(k)$. Suppose the following conditions:
    \begin{enumerate}
        \item The induced homomorphism $\pi(\mathcal{C}_X,x)\rightarrow \pi(\mathcal{C}_S,s)$ is faithfully flat.
        \item For any $E\in\mathcal{C}_X$, we have $f_*E\in\mathcal{C}_S$.
    \end{enumerate}
    Then the canonical map $f^*f_*E\rightarrow E$ is an embedding in $\mathcal{C}_X$ for any $E\in\mathcal{C}_X$.
\end{Proposition}

\begin{proof}
    Let $\pi:P\rightarrow X$ be the principal $\pi(\mathcal{C}_X,x)$-bundle corresponding to $\mathcal{C}_X$, $\pi':P'\rightarrow S$ the principal $\pi(\mathcal{C}_S,s)$-bundle corresponding to $\mathcal{C}_S$. Then we have the following fully faithful exact tensor functors
    \[
        \begin{aligned}
            \eta_X^P&:\Rep_k^f(\pi(\mathcal{C}_X,x))\rightarrow \mathcal{C}_X,V_X\mapsto (V_X\otimes_k \mathcal{O}_P)^{\pi(\mathcal{C}_X,x)},\eta_S^{P'}:\Rep_k^f(\pi(\mathcal{C}_S,x))\rightarrow \mathcal{C}_S,V_S\mapsto (V_S\otimes_k \mathcal{O}_{P'})^{\pi(\mathcal{C}_S,s)}.
        \end{aligned}
    \]
    Denote $\varphi$ the homomoprhism $\pi(\mathcal{C}_X,x)\rightarrow \pi(\mathcal{C}_S,s)$. Consider the following commutative diagram
    \[
        \begin{tikzcd}
            \Rep_k^f(\pi(\mathcal{C}_S,s))\arrow[r,"\varphi^*"]\arrow[d,"\eta_S^{P'}"]&\Rep_k^f(\pi(\mathcal{C}_X,x))\arrow[d,"\eta_X^P"]\\
            \mathcal{C}_S\arrow[r,"f^*"]&\mathcal{C}_X
        \end{tikzcd}
    \]
    i.e. $f^*(\eta_S^{P'}(V_S))\cong \eta_X^P(\varphi^*V_S)$ for any $V_S\in \Rep_k^f(\pi(\mathcal{C}_S,s))$. Let $\ker$ be the kernel of $\varphi:\pi(\mathcal{C}_X,x)\rightarrow \pi(\mathcal{C}_S,s)$. By Lemma~\ref{adjoint}, $-^{\ker}:\Rep_k^f (\pi(\mathcal{C}_X,x))\rightarrow \Rep_k^f (\pi(\mathcal{C}_S,s))$ is the right adjoint of $\varphi^*:\Rep_k^f (\pi(\mathcal{C}_S,s))\rightarrow\Rep_k^f (\pi(\mathcal{C}_X,x))$. For any $V_S\in \Rep_k^f(\pi(\mathcal{C}_S,s))$ and any $V_X\in\Rep_k^f(\pi(\mathcal{C}_X,x))$, we have
    $$\begin{aligned}
        \Hom_{\Rep_k^f(\pi(\mathcal{C}_S,s))}(V_S,V_X^{\ker})&\cong \Hom_{\Rep_k^f(\pi(\mathcal{C}_X,x))}(\varphi^* V_S,V_X)
    \end{aligned}$$
    Since the functors $\eta_X^P$ and $\eta_S^{P'}$ are fully faithful, we have
    $$\begin{aligned}
        \Hom_{\mathcal{O}_S}(\eta_S^{P'}(V_S),\eta_S^{P'}(V_X^{\ker}))&\cong \Hom_{\mathcal{O}_X}(\eta_X^P(\varphi^*V_S),\eta_X^P(V_X))\\
        &\cong \Hom_{\mathcal{O}_X}(f^*(\eta_S^{P'}(V_S)),\eta_X^P(V_X))\cong \Hom_{\mathcal{O}_S}(\eta_S^{P'}(V_S),f_*(\eta_X^P(V_X))).
    \end{aligned}$$
    Then we have the following commutative diagram
    \[
        \begin{tikzcd}
            \Rep_k^f(\pi(\mathcal{C}_S,s))\arrow[d,"\eta_S^{P'}"]&\Rep_k^f(\pi(\mathcal{C}_X,x))\arrow[l,"-^{\ker}",swap]\arrow[d,"\eta_X^P"]\\
            \mathcal{C}_S&\mathcal{C}_X\arrow[l,"f_*",swap]
        \end{tikzcd}
    \]
    i.e. $f_*(\eta_X^P(V_X))\cong \eta_S^{P'}(V_X^{\ker})$ for any $V_X\in \Rep_k^f(\pi(\mathcal{C}_X,x))$. Let $V_X\in\Rep_k^f(\pi(\mathcal{C}_X,x))$, $E=\eta_X^P(V_X)\in\mathcal{C}_X$, then we have $f_*E=f_*(\eta_X^P(V_X))\cong \eta_S^{P'}(V_X^{\ker})$. Consider the following commutative diagram
    \[
        \begin{tikzcd}
            \Hom_S(f_*E,f_*E)\arrow[r,"\cong"]\arrow[d,"\cong"]& \Hom_X(f^*f_*E, E)\arrow[d]\\
            \Hom_{\pi(\mathcal{C}_S,s)}(V_X^{\ker},V_X^{\ker})\arrow[r,"\cong"]&\Hom_{\pi(\mathcal{C}_X,x)}(\varphi^*V_X^{\ker},V_X)
        \end{tikzcd}
    \]
    Then we have $\Hom_X(f^*f_*E, E)\cong \Hom_{\pi(\mathcal{C}_X,x)}(\varphi^*V_X^{\ker},V_X)$. Consider the correspondence of canonical map
    $$\begin{aligned}
      f^*f_*E\rightarrow E ~~\longleftrightarrow~~ \varphi^*V_X^{\ker} \hookrightarrow V_X
    \end{aligned}$$
    Then the canonical map $f^*f_*E\rightarrow E$ is an embedding in $\mathcal{C}_X$.
\end{proof}

\begin{Proposition}\label{equivprin}
    Let $k$ be a field, $f:X\rightarrow S$ a morphism of connected schemes proper over field $k$, $x\in X(k)$ lying over $s\in S(k)$. Denote 
    \begin{enumerate}
        \item[(a)] $\phi:P\rightarrow X$ the principal $\pi(\mathcal{C}_X,x)$-bundle corresponding to $\mathcal{C}_X$.
        \item[(b)] $\phi':P'\rightarrow S$ the principal $\pi(\mathcal{C}_S,s)$-bundle corresponding to $\mathcal{C}_S$.
    \end{enumerate}
    Suppose the following conditions:
    \begin{enumerate}
        \item[I.] The induced homomorphism $\pi(\mathcal{C}_X,x)\twoheadrightarrow \pi(\mathcal{C}_S,s)$ is faithfully flat.
        \item[II.] For any $E\in\mathcal{C}_X$, we have $f_*E\in\mathcal{C}_S$.
    \end{enumerate}
    For any filtered colimit system $\{\mathcal{E}_\alpha\}_{\alpha\in I}$ of subsheaves of $\phi_*\mathcal{O}_P$, s.t. $\phi_*\mathcal{O}_P=\varinjlim\mathcal{E}_{\alpha}$ and $\mathcal{E}_{\alpha}\in \mathcal{C}_X$. Then
    \begin{enumerate}
        \item $\langle \{f_* \mathcal{E}_{\alpha} \}_{\alpha\in I}\rangle_{\mathcal{C}_S}=\mathcal{C}_S$.
        \item The canonical map $f^*{\phi'}_*\mathcal{O}_{P'}\rightarrow f^*f_*\phi_*\mathcal{O}_P$ induced by $\pi(\mathcal{C}_X,x)\rightarrow \pi(\mathcal{C}_S,s)$ is an isomorphism.
    \end{enumerate}
\end{Proposition}
\begin{proof}
We have the following fully faithful exact tensor functors
    \[
        \begin{aligned}
            \eta_S^{P'}&:\Rep_k^f(\pi(\mathcal{C}_S,s))\rightarrow \mathcal{C}_S,V_S\mapsto (V_S\otimes_k \mathcal{O}_{P'})^{\pi(\mathcal{C}_S,s)},\\
            \eta_X^P&:\Rep_k^f(\pi(\mathcal{C}_X,x))\rightarrow \mathcal{C}_X,V_X\mapsto (V_X\otimes_k \mathcal{O}_P)^{\pi(\mathcal{C}_X,x)}.
        \end{aligned}
    \]
By \cite[Proposition~2.9]{Nor76}, we have the principal $\pi(\mathcal{C}_S,s)$-bundle $f^*P'\rightarrow X$ and it corresponds to the functor
\[
        \begin{aligned}
            \eta_X^{f^*P'}&:\Rep_k^f(\pi(\mathcal{C}_S,x))\rightarrow \mathcal{C}_X,V_S\mapsto f^*\eta_S^{P'}(V_S).\\
        \end{aligned}
    \]
The composite $(\eta_X^P)^{-1} \circ \eta_X^{f^*P'}$ coincides with the functor induced by the homomorphism $\pi(\mathcal{C}_X,x)\twoheadrightarrow \pi(\mathcal{C}_S,s)$. Since the induced homomorphism $\pi(\mathcal{C}_X,x)\twoheadrightarrow \pi(\mathcal{C}_S,s)$ is faithfully flat, by Theorem~\ref{Thm1}, the functor $f^*$ is fully faithful and $\eta_X^{f^*P'}=\eta_X^P\circ f^*$ is fully faithful. Then $f^*P'$ is the principal $\pi(\mathcal{C}_S,s)$-bundle induced from $P$ by the faithfully flat homomorphism $\pi(\mathcal{C}_X,x)\rightarrow \pi(\mathcal{C}_S,s)$. So we obtain a faithfully flat morphism $P\twoheadrightarrow f^*P'$ and $f^*{\phi'}_*\mathcal{O}_{P'}\hookrightarrow \phi_*\mathcal{O}_P$. Since $f^*{\phi'}_*\mathcal{O}_{P'}\hookrightarrow \phi_*\mathcal{O}_P$ factors through 
$$f^*{\phi'}_*\mathcal{O}_{P'}\rightarrow f^*f_*\phi_*\mathcal{O}_P \rightarrow \phi_*\mathcal{O}_P,$$
we obtain a canonical embedding $f^*{\phi'}_*\mathcal{O}_{P'}\hookrightarrow f^*f_*\phi_*\mathcal{O}_P$.
 
For any $E\in\mathcal{C}_X$, we have $f_*E\in\mathcal{C}_S$, so $\{f_*\mathcal{E}_{\alpha}\}_{\alpha\in I}\subseteq \mathcal{C}_S$. We denote the Tannaka group scheme of the Tannakian category $\langle\{f_*\mathcal{E}_{\alpha}\}_{\alpha\in I}\rangle_{\mathcal{C}_S}$ by $G':=\pi(\langle\{f_*\mathcal{E}_{\alpha}\}_{\alpha\in I}\rangle_{\mathcal{C}_S},s)$ and the principal $G'$-bundle corresponding to $\langle\{f_*\mathcal{E}_{\alpha}\}_{\alpha\in I}\rangle_{\mathcal{C}_S}$ by $\phi'': P''\rightarrow S$. Then we have a faithfully flat morphism $P'\twoheadrightarrow P''$ since $P''$ is induced from $P'$ by the faithfully flat homomorphism $\pi(\mathcal{C}_S,s)\twoheadrightarrow G'$. It follows that ${\phi''}_*\mathcal{O}_{P''}\hookrightarrow {\phi'}_*\mathcal{O}_{P'}$, which corresponds to $k[G']\hookrightarrow k[\pi(\mathcal{C}_S,s)]\in\Rep_k(\pi(\mathcal{C}_S,s))$. Since the faithfully flat homomorphism $\pi(\mathcal{C}_X,x)\rightarrow \pi(\mathcal{C}_S,s)$ induces an exact functor $\Rep_k(\pi(\mathcal{C}_S,s))\rightarrow \Rep_k(\pi(\mathcal{C}_X,x))$, we have
$$f^*{\phi''}_*\mathcal{O}_{P''}\hookrightarrow f^*{\phi'}_*\mathcal{O}_{P'}\hookrightarrow f^*f_*\phi_*\mathcal{O}_P\rightarrow\phi_*\mathcal{O}_P.$$

Denote $E:={\phi}_*\mathcal{O}_{P}$, $E':={\phi'}_*\mathcal{O}_{P'}$ and $E'':=\phi'_*\mathcal{O}_{P''}$. By I, II and Proposition~\ref{embedding}, for any $F\in\mathcal{C}_X$ the canonical map $f^*f_*F\rightarrow F$ is an embedding in $\mathcal{C}_X$. In particular, we have $f^*f_*\mathcal{E}_\alpha\hookrightarrow \mathcal{E}_\alpha$ for any $\alpha\in I$. Taking colimit and by \cite[\href{https://stacks.math.columbia.edu/tag/07U6}{Tag 07U6}]{stacks}, we obtain a canonical embedding $f^*f_*E\hookrightarrow E$. Then we have 
$$f^*E''\hookrightarrow f^*E'\hookrightarrow f^*f_*E\hookrightarrow E,$$
which corresponds to an embedding 
$$k[G']\hookrightarrow k[\pi(\mathcal{C}_S,s)]\hookrightarrow V\hookrightarrow k[\pi(\mathcal{C}_X,x)]\in\Rep_k(\pi(\mathcal{C}_X,x)),$$
where $k[G']$ and $V\in\Rep_k(\pi(\mathcal{C}_X,x))$ come from $\Rep_k(G')$. 

Let $H$ be the kernel of $\pi(\mathcal{C}_X,x)\twoheadrightarrow \pi(\mathcal{C}_S,s)\twoheadrightarrow G'$. Then $V^H=\{v\in V\vert \rho(h)v=v,\forall h\in H\}=V$ and 
$$k[G']^H=k[G']\hookrightarrow k[\pi(\mathcal{C}_S,s)]^H \hookrightarrow V^H=V\hookrightarrow k[\pi(\mathcal{C}_X,x)]^H\in\Rep_k(\pi(\mathcal{C}_X,x)).$$
Since $k[G']=k[\pi(\mathcal{C}_X,x)]^H$ by \cite[Chapter~16.3]{Wat79}, we have $k[G']=V$. This implies $k[G']=k[\pi(\mathcal{C}_S,s)]=V$. Then the faithfully flat homomorphism $\pi(\mathcal{C}_S,s)\rightarrow G'$ is an isomorphism. Hence $\langle \{f_* \mathcal{E}_{\alpha\in I} \}_{\alpha}\rangle_{\mathcal{C}_S}=\mathcal{C}_S$. Moreover, we have 
$$f^*E''\cong f^*E'\cong f^*f_*E,\text{ i.e. }
f^*{\phi''}_*\mathcal{O}_{P''}\cong f^*{\phi'}_*\mathcal{O}_{P'}\cong f^*f_*\phi_*\mathcal{O}_P.$$
\end{proof}

\begin{Proposition}\label{equivpushforward}
    Let $k$ be a field, $f:X\rightarrow S$ a morphism of connected schemes proper over $k$, $x\in X(k)$ lying over $s\in S(k)$. Denote $\phi:P\rightarrow X$ the principal $\pi(\mathcal{C}_X,x)$-bundle corresponding to $\mathcal{C}_X$. Then the following conditions are equivalent:
\begin{enumerate}
    \item For any $E\in\mathcal{C}_X$, $f_*E\in\mathcal{C}_S$.
    \item Let $M\subseteq\mathcal{C}_X$, $\phi_M:P_M\rightarrow X$ be the principal $\pi(\langle M\rangle_{\mathcal{C}_X}, x)$-bundle corresponding to $\langle M\rangle_{\mathcal{C}_X}$. For any filtered colimit systems $\{{\mathcal{E}_M}_\alpha\}_{\alpha\in I}$ in $\mathcal{C}_X$, s.t. ${\phi_M}_*\mathcal{O}_{P_M}=\varinjlim{\mathcal{E}_M}_\alpha$, we have $f_*{\mathcal{E}_M}_\alpha\in\mathcal{C}_S$ for any $\alpha$.
    \item Let $M\subseteq\mathcal{C}_X$, $\phi_M:P_M\rightarrow X$ the principal $\pi(\langle M\rangle_{\mathcal{C}_X}, x)$-bundle corresponding to $\langle M\rangle_{\mathcal{C}_X}$. There is a filtered colimit system $\{{\mathcal{E}_M}_\alpha\}_{\alpha\in I}$ in $\mathcal{C}_X$, s.t. ${\phi_M}_*\mathcal{O}_{P_M}=\varinjlim{\mathcal{E}_M}_\alpha$ and $f_*{\mathcal{E}_M}_\alpha\in\mathcal{C}_S$ for any $\alpha$.
    \item For any filtered colimit systems $\{\mathcal{E}_\alpha\}_{\alpha\in I}$ in $\mathcal{C}_X$, s.t. $\phi_*\mathcal{O}_P=\varinjlim\mathcal{E}_{\alpha}$, we have $f_*\mathcal{E}_\alpha\in\mathcal{C}_S$ for any $\alpha$.
    \item There is a filtered colimit system $\{\mathcal{E}_\alpha\}_{\alpha\in I}$ in $\mathcal{C}_X$, s.t. $\phi_*\mathcal{O}_P=\varinjlim\mathcal{E}_{\alpha}$ and $f_*\mathcal{E}_\alpha\in\mathcal{C}_S$ for any $\alpha$.
\end{enumerate}
\end{Proposition}
\begin{proof}
$(1)\Rightarrow (2)$ Obviously.

$(2)\Rightarrow (3)$ Obviously.

$(4)\Rightarrow (5)$ Obviously.

$(2)\Rightarrow (4)$ Take $M=\mathcal{C}_X$, it follows immediately.

$(3)\Rightarrow (5)$ Take $M=\mathcal{C}_X$, it follows immediately.

$(5)\Rightarrow (1)$ Suppose there is a filtered colimit system $\{\mathcal{E}_\alpha\}_{\alpha\in I}$ in $\mathcal{C}_X$, s.t. $\phi_*\mathcal{O}_P=\varinjlim\mathcal{E}_{\alpha}$ and $f_*\mathcal{E}_\alpha\in\mathcal{C}_S$ for any $\alpha$. Let $E\in\mathcal{C}_X$, denote its corresponding $\pi(\mathcal{C}_X,x)$-representation by $V$. By Lemma~\ref{regularrepresentation}, there exists some positive integer $n$ such that $V\hookrightarrow k[\pi(\mathcal{C}_X,x)]^{\oplus n}$. Then there exists an inclusion 
$$E\hookrightarrow \mathcal{E}_{\alpha_1}\oplus\cdots\oplus \mathcal{E}_{\alpha_n},$$
where $\alpha_i\in I$ for any $i$. Consider the following exact sequence
$$0\rightarrow E\hookrightarrow \mathcal{E}_{\alpha_1}\oplus\cdots\oplus \mathcal{E}_{\alpha_n}\rightarrow \coker\rightarrow 0.$$
Note that $\coker\in\mathcal{C}_X$, applying Lemma~\ref{regularrepresentation} again, we obtain $\coker\hookrightarrow \mathcal{E}_{\alpha_1'}\oplus\cdots\oplus \mathcal{E}_{\alpha_m'}$ and an exact sequence
$$0\rightarrow E\rightarrow \mathcal{E}_{\alpha_1}\oplus\cdots\oplus \mathcal{E}_{\alpha_n}\rightarrow \mathcal{E}_{\alpha_1'}\oplus\cdots\oplus \mathcal{E}_{\alpha_m'}.$$
where $\alpha'_j\in I$ for any $j$. Applying the left exact functor $f_*$, we obtain an exact sequence
$$0\rightarrow f_*E\rightarrow f_*\mathcal{E}_{\alpha_1}\oplus\cdots\oplus f_*\mathcal{E}_{\alpha_n}\rightarrow f_*\mathcal{E}_{\alpha_1'}\oplus\cdots\oplus f_*\mathcal{E}_{\alpha_m'}.$$
By assumption, we have $f_*\mathcal{E}_\alpha\in\mathcal{C}_S$ for any $\alpha\in I$, thus $f_*\mathcal{E}_{\alpha_1}\oplus\cdots\oplus f_*\mathcal{E}_{\alpha_n}\in\mathcal{C}_S$ and $f_*\mathcal{E}_{\alpha_1'}\oplus\cdots\oplus f_*\mathcal{E}_{\alpha_m'}\in\mathcal{C}_S$. Then we have $f_*E\in\mathcal{C}_S$ since $\mathcal{C}_S$ is an abelian category.
\end{proof}

\begin{Proposition}\label{equivbasechange}
    Let $k$ be a field, $f:X\rightarrow S$ a morphism of connected schemes proper over $k$, $x\in X(k)$ lying over $s\in S(k)$. Denote $\phi:P\rightarrow X$ the principal $\pi(\mathcal{C}_X,x)$-bundle corresponding to $\mathcal{C}_X$. Then the following conditions are equivalent.
    \begin{enumerate}
        \item For any $E\in\mathcal{C}_X$, we have $f_*E\in\mathcal{C}_S$ and $E$ satisfies base change at $s$.
        \item Let $M\subseteq\mathcal{C}_X$, $\phi_M:P_M\rightarrow X$ the principal $\pi(\langle M\rangle_{\mathcal{C}_X}, x)$-bundle corresponding to $\langle M\rangle_{\mathcal{C}_X}$. For any filtered colimit systems $\{{\mathcal{E}_M}_\alpha\}_{\alpha\in I}$ in $\mathcal{C}_X$, s.t. ${\phi_M}_*\mathcal{O}_{P_M}=\varinjlim{\mathcal{E}_M}_\alpha$, we have $f_*{\mathcal{E}_M}_\alpha\in\mathcal{C}_S$ and ${\mathcal{E}_M}_\alpha$ satisfies base change at $s$ for any $\alpha$.
    \item Let $M\subseteq\mathcal{C}_X$, $\phi_M:P_M\rightarrow X$ the principal $\pi(\langle M\rangle_{\mathcal{C}_X}, x)$-bundle corresponding to $\langle M\rangle_{\mathcal{C}_X}$. There is a filtered colimit system $\{{\mathcal{E}_M}_\alpha\}_{\alpha\in I}$ in $\mathcal{C}_X$, s.t. ${\phi_M}_*\mathcal{O}_{P_M}=\varinjlim{\mathcal{E}_M}_\alpha$, $f_*{\mathcal{E}_M}_\alpha\in\mathcal{C}_S$ and ${\mathcal{E}_M}_\alpha$ satisfies base change at $s$ for any $\alpha$.
    \item For any filtered colimit systems $\{\mathcal{E}_\alpha\}_{\alpha\in I}$ in $\mathcal{C}_X$, s.t. $\phi_*\mathcal{O}_P=\varinjlim\mathcal{E}_{\alpha}$, we have $f_*\mathcal{E}_\alpha\in\mathcal{C}_S$ and $\mathcal{E}_\alpha$ satisfies base change at $s$ for any $\alpha$.
    \item There is a filtered colimit system $\{\mathcal{E}_\alpha\}_{\alpha\in I}$ in $\mathcal{C}_X$, s.t. $\phi_*\mathcal{O}_P=\varinjlim\mathcal{E}_{\alpha}$, $f_*\mathcal{E}_\alpha\in\mathcal{C}_S$ and $\mathcal{E}_\alpha$ satisfies base change at $s$ for any $\alpha$.
    \end{enumerate}
\end{Proposition}

\begin{proof}
    $(1)\Rightarrow (2)$ Obviously.

$(2)\Rightarrow (3)$ Obviously.

$(4)\Rightarrow (5)$ Obviously.

$(2)\Rightarrow (4)$ Take $M=\mathcal{C}_X$, it follows immediately.

$(3)\Rightarrow (5)$ Take $M=\mathcal{C}_X$, it follows immediately.

$(5)\Rightarrow (1)$ Suppose there is a filtered colimit system $\{\mathcal{E}_\alpha\}_{\alpha\in I}$ in $\mathcal{C}_X$, s.t. $\phi_*\mathcal{O}_P=\varinjlim\mathcal{E}_{\alpha}$, $f_*\mathcal{E}_\alpha\in\mathcal{C}_S$ and $\mathcal{E}_\alpha$ satisfies base change at $s$ for any $\alpha$. By Proposition~\ref{equivpushforward}, for any $E\in\mathcal{C}_X$, we have $f_*E\in\mathcal{C}_S$. 

Let $E\in\mathcal{C}_X$, denote its corresponding $\pi(\mathcal{C}_X,x)$-representation by $V$. By Lemma~\ref{regularrepresentation}, there exists some positive integer $n$ such that $V\hookrightarrow k[\pi(\mathcal{C}_X,x)]^{\oplus n}$. Then there exists an inclusion 
$$E\hookrightarrow \mathcal{E}_{\alpha_1}\oplus\cdots\oplus \mathcal{E}_{\alpha_n}.$$
where $\alpha_i\in I$ for any $i$. Consider the following exact sequence
$$0\rightarrow E\hookrightarrow \mathcal{E}_{\alpha_1}\oplus\cdots\oplus \mathcal{E}_{\alpha_n}\rightarrow \coker\rightarrow 0.$$
Note that $\coker\in\mathcal{C}_X$, applying Lemma~\ref{regularrepresentation} again, we obtain $\coker\hookrightarrow \mathcal{E}_{\alpha_1'}\oplus\cdots\oplus \mathcal{E}_{\alpha_m'}$ and an exact sequence
$$0\rightarrow E\rightarrow \mathcal{E}_{\alpha_1}\oplus\cdots\oplus \mathcal{E}_{\alpha_n}\rightarrow \mathcal{E}_{\alpha_1'}\oplus\cdots\oplus \mathcal{E}_{\alpha_m'}.$$
where $\alpha'_j\in I$ for any $j$. Applying the left exact functor $f_*$, we obtain an exact sequence
$$0\rightarrow f_*E\rightarrow f_*(\mathcal{E}_{\alpha_1}\oplus\cdots\oplus \mathcal{E}_{\alpha_n})\rightarrow f_*(\mathcal{E}_{\alpha_1'}\oplus\cdots\oplus \mathcal{E}_{\alpha_m'}).$$
Note that $f_*E,f_*(\mathcal{E}_{\alpha_1}\oplus\cdots\oplus \mathcal{E}_{\alpha_n}),f_*(\mathcal{E}_{\alpha_1'}\oplus\cdots\oplus \mathcal{E}_{\alpha_m'})\in\mathcal{C}_S$.  Taking fibre at $s$, we have an exact sequence
$$0\rightarrow f_*E|_s\rightarrow f_*(\mathcal{E}_{\alpha_1}\oplus\cdots\oplus \mathcal{E}_{\alpha_n})|_s\rightarrow f_*(\mathcal{E}_{\alpha_1'}\oplus\cdots\oplus \mathcal{E}_{\alpha_m'})|_s.$$
Since $\mathcal{E}_\alpha$ satisfies base change at $s$, we have the following commutative diagram
\[
    \begin{tikzcd}[sep=tiny]
        &0\arrow[r] & f_*E|_s\arrow[rr]\arrow[ld]\arrow[dd,equal] &&\oplus (f_*\mathcal{E}_{\alpha_i}|_s)\arrow[rr]\arrow[ld,"\cong"]\arrow[dd,"\cong"near start] &&\oplus (f_*\mathcal{E}_{\alpha'_i}|_s)\arrow[ld,"\cong"]\arrow[dd,"\cong"near start]\\
        0\arrow[r]&H^0(X_s, E|_{X_s})\arrow[rr,crossing over]&& \oplus H^0(X_s,\mathcal{E}_{\alpha_i}|_{X_s})\arrow[rr,crossing over]&& \oplus H^0(X_s,\mathcal{E}_{\alpha'_i}|_{X_s})&\\
        &0\arrow[r] & f_*E|_s\arrow[rr]\arrow[ld] &&f_*(\oplus \mathcal{E}_{\alpha_i})|_s\arrow[rr]\arrow[ld,"\cong"] &&f_*(\oplus \mathcal{E}_{\alpha'_i})|_s\arrow[ld,"\cong"]\\
        0\arrow[r]&H^0(X_s, E|_{X_s})\arrow[rr]\arrow[from=uu,equal,bend right]&&H^0(X_s,(\oplus\mathcal{E}_{\alpha_i})|_{X_s})\arrow[rr]\arrow[from=uu,crossing over,"\cong"near start]&&H^0(X_s,(\oplus\mathcal{E}_{\alpha'_i})|_{X_s})\arrow[from=uu,crossing over,"\cong"near start]&
    \end{tikzcd}
\]
Then by five lemma, $f_*E|_s\cong H^0(X_s, E|_{X_s})$.
\end{proof}

Now we state our main result in this section as follows:

\begin{Theorem}\label{Main}
Let $k$ be a field, $f:X\rightarrow S$ a morphism of connected schemes proper over $k$ with $f_*\mathcal{O}_X\cong\mathcal{O}_S$, $x\in X(k)$ lying over $s\in S(k)$, $\mathcal{C}_X,\mathcal{C}_S,\mathcal{C}_{X_s}$ the Tannakian categories over $X,S,X_s$ respectively. Suppose the pullback induces functors $\mathcal{C}_S\xrightarrow{f^*}\mathcal{C}_X\xrightarrow{|_{X_s}}\mathcal{C}_{X_s}$. Then the following conditions are equivalent:
\begin{itemize}
    \item[(1)] \begin{itemize}
            \item[i.] The functor $|_{X_s}:\mathcal{C}_X\rightarrow \mathcal{C}_{X_s}$ is observable;
            \item[ii.] For any $E\in \mathcal{C}_X$, $f_* E\in\mathcal{C}_S$ and $E$ satisfies base change at $s$.
        \end{itemize} 
	\item[(2)]  \begin{itemize}
            \item[i.] The functor $|_{X_s}:\mathcal{C}_X\rightarrow \mathcal{C}_{X_s}$ is observable;
            \item[ii.] For any subset $M\subseteq\mathcal{C}_X$ 
             and its corresponding principal $\pi(\langle M\rangle_{\mathcal{C}_X},x)$-principal bundle $\phi_M: P_M\rightarrow X$, there exists a $($for any$)$ filtered colimit system $\{{\mathcal{E}_M}_{\alpha}\}_{\alpha\in I}$ of subsheaves of ${\phi_M}_*\mathcal{O}_{P_M}$, s.t. ${\phi_M}_*\mathcal{O}_{P_M}=\varinjlim{\mathcal{E}_M}_{\alpha}$ where ${\mathcal{E}_M}_{\alpha}\in \mathcal{C}_X$, $f_* {\mathcal{E}_M}_{\alpha}\in\mathcal{C}_S$ and ${\mathcal{E}_M}_{\alpha}$ satisfies base change at $s$.
                \end{itemize}
    \item[(3)] \begin{itemize}
            \item[i.] The functor $|_{X_s}:\mathcal{C}_X\rightarrow \mathcal{C}_{X_s}$ is observable;
            \item[ii.] Let $\phi: P\rightarrow X$ be the principal $\pi(\mathcal{C}_X,x)$-bundle corresponding to $\mathcal{C}_X$. There exists a $($for any$)$ filtered colimit system $\{\mathcal{E}_\alpha\}_{\alpha\in I}$ of subsheaves of $\phi_*\mathcal{O}_P$, s.t. $\phi_*\mathcal{O}_P=\varinjlim\mathcal{E}_{\alpha}$ where $\mathcal{E}_{\alpha}\in \mathcal{C}_X$, $f_* \mathcal{E}_{\alpha}\in\mathcal{C}_S$ and $\mathcal{E}_{\alpha}$ satisfies base change at $s$.
            \end{itemize}
\end{itemize}
The above equivalent conditions imply the following exact homotopy sequence $$\pi(\mathcal{C}_{X_s},x)\rightarrow\pi(\mathcal{C}_X,x)\rightarrow\pi(\mathcal{C}_S,s)\rightarrow 1$$
Moreover, if we suppose further that $S$ is integral and $f$ is a flat morphism, then the exactness of homotopy sequence also implies the above three equivalent conditions.
\end{Theorem}

\begin{proof}
The equivalence between condition (1), (2) and (3) follows by Proposition~\ref{equivbasechange}.

$(1)\Rightarrow \mathrm{Exactness}.$
For any $F\in\mathcal{C}_S$ and any subobject $E\hookrightarrow f^*F\in\mathcal{C}_X$, we have $f_*F\in\mathcal{C}_S$ and $F$ satisfies base change at $s$ by ii. Since $E|_{X_s}\hookrightarrow f^*F|_{X_s}\equiv F|_s\otimes_k \mathcal{O}_{X_s}\in\mathcal{C}_{X_s}$ and any subobject of a trivial object is also trivial, $E|_{X_s}$ is trivial. By Proposition~\ref{generalsur}, the homomorphism $\pi(\mathcal{C}_X,x)\rightarrow\pi(\mathcal{C}_S,s)$ is faithfully flat.

Let $E\in\mathcal{C}_X$, then we have $f_*E\in\mathcal{C}_S$ and $E$ satisfies base change at $s$ by ii. Then $f^*f_*E\in\mathcal{C}_X$. By Proposition~\ref{MaximalTrivial}, $f^*f_*E \hookrightarrow E$ is a subobject in $\mathcal{C}_X$. Then
$$f^*f_*E|_{X_s}\equiv f_*E|_s\otimes_k \mathcal{O}_{X_s}\cong H^0(X_s,E|_{X_s})\otimes_k \mathcal{O}_{X_s}$$
is the maximal trivial subobject of $E|_{X_s}$. Hence condition (b) in [Theorem~\ref{Thm1}, (3)] is satisfied.

Suppose $E\in\mathcal{C}_X$ and $E|_{X_s}$ is a trivial object in $\mathcal{C}_{X_s}$. Since $E$ satisfies base change at $s$ by ii, we have $f^*f_*E\hookrightarrow E$ is a subobject in $\mathcal{C}_X$ by Proposition~\ref{MaximalTrivial}. Consider the following exact sequence in $\mathcal{C}_X$:
$$0\rightarrow f^*f_* E\rightarrow E\rightarrow E/f^*f_* E \rightarrow 0.$$
Restricting it to the fibre $X_s$, we obtain an exact sequence in $\mathcal{C}_{X_s}$:
$$0\rightarrow (f^*f_*E)\vert_{X_s}\rightarrow E\vert_{X_s}\rightarrow (E/f^*f_* E)\vert_{X_s}\rightarrow 0.$$
Since $E$ satisfies base change at $s$, we have $$E\vert_{X_s}=H^0(X_s,E\vert_{X_s})\otimes_k\mathcal{O}_{X_s}\cong f^*f_*E|_{X_s}.$$
So $(E/f^*f_* E)\vert_{X_s}=0$. It follows that $E/f^*f_* E=0$. Then $E\cong f^*f_*E$ where $f_*E\in\mathcal{C}_S$. Hence the condition (a) in [Theorem~\ref{Thm1}, (3)] is satisfied. 

The condition (c) in [Theorem~\ref{Thm1}, (3)] is just our condition i by Lemma~\ref{observable}.

Consequently, we have an exact homopoty sequence
$$\pi(\mathcal{C}_{X_s},x)\rightarrow\pi(\mathcal{C}_X,x)\rightarrow\pi(\mathcal{C}_S,s)\rightarrow 1.$$

$\mathrm{Exactness}\Rightarrow (1)$ The condition i follows by [Theorem~\ref{Thm1}, (3)] and Lemma~\ref{observable}.

For any $E\in \mathcal{C}_X$, $H^0(X_s,E\vert_{X_s})\otimes_k\mathcal{O}_{X_s}\hookrightarrow E\vert_{X_s}$ is the maximal trivial subobject. Since the homotopy sequence is exact, by Theorem~\ref{Thm1}, there exists $F\hookrightarrow E$ such that $F\vert_{X_s}=H^0(X_s,E\vert_{X_s})\otimes\mathcal{O}_{X_s}$ and there exists $F_0\in \mathcal{C}_S$ such that $F=f^*F_0$. 

Note that $F=f^*F_0\hookrightarrow E$ factors through $F\rightarrow f^*f_*E\rightarrow E$. 
Restricting to the fibre $X_s$ and taking global sections, we obtain
$$H^0(X_s,F|_{X_s})\rightarrow H^0(X_s,f^*f_*E\vert_{X_s})\rightarrow H^0(X_s,E\vert_{X_s}).$$
Since $H^0(X_s,F\vert_{X_s})\cong H^0(X_s,E\vert_{X_s})$, so the natural homomorphism 
$$(f_*E)\vert_{s}\equiv H^0(X_s,(f^*f_* E)\vert_{X_s})\twoheadrightarrow H^0(X_s,E\vert_{X_s})$$
is surjective. Then by \cite[Theorem 25.1.6]{Vak25}, we have 
$$(f_*E)\vert_{s}\cong H^0(X_s,E\vert_{X_s}).$$
Hence $E$ satisfies base change at $s$.

By Proposition~\ref{generalsur}, $F$ satisfies base change at any $t\in S$. Since $S$ is integral, there exists an open neighborhood $U$ of $s$ such that $f_*E$ and $f_*F$ are vector bundle on $U$ of the same rank accoring to \cite[Theorem 25.1.6]{Vak25}. Then $F\hookrightarrow E$ induces an isomorphism between $f_*F$ and $f_* E$ on $U$. Applying the left exact functor $f_*$ to the exact sequence $0\rightarrow F\rightarrow E\rightarrow E/F\rightarrow 0$ in $\mathcal{C}_X$, we obtain
$$0\rightarrow f_*F\rightarrow f_*E\rightarrow f_*(E/F).$$
Consider the following commutative diagram
\[\begin{tikzcd}
    0\arrow[r]&f_*F\arrow[r]&f_*E\arrow[r]\arrow[rd]&f_*(E/F)\\
    &&&\coker\arrow[u,hook]
\end{tikzcd}\]
Since $E/F\in\mathcal{C}_X$ is a vector bundle, $S$ is integral and $f$ is flat, we have $f_*(E/F)$ is torsion free on $S$. So $\coker$ is also torsion free on $S$. But $f_*F$ and $f_*E$ is isomorphic on $U$, it follows that $\coker|_{U}=0$. Therefore $\coker=0$ on $S$. Then $f_*F\cong f_*E$ on $X$. Hence $$f_*E\cong f_*F=f_*f^*F_0\cong f_*\mathcal{O}_X\otimes_{\mathcal{O}_{S}}F_0\cong F_0\in\mathcal{C}_S.$$
\end{proof}

\section{The K\"{u}nneth Formula of Fundamental Group Schemes}

\begin{Lemma}\label{split}
    Let $k$ be a field, $A\xrightarrow{\alpha} B\xrightarrow{\beta} C$ a sequence of affine group schemes over $k$. Then the following conditions are equivalent:
    \begin{enumerate}
        \item \begin{enumerate}
            \item The sequence $A\xrightarrow{\alpha} B\xrightarrow{\beta} C$ is exact at $B$.
            \item There exist homomorphisms $\sigma:B\rightarrow A$ and $\tau:C\rightarrow B$ such that $\sigma\circ \alpha =\Id_A$ and $\beta\circ \tau =\Id_C$.
        \end{enumerate}
        \item There exists an isomorphism $\varphi:B\rightarrow A\times_k C$ such that $\varphi\circ \alpha=i_A$ and $p_C\circ \varphi =\beta$, where $i_A:A\rightarrow A\times_k C$ is the natural embedding and $p_C:A\times_k C\rightarrow C$ is the natural projection.
    \end{enumerate}
\end{Lemma}

\begin{proof}
    $(1)\Rightarrow (2)$ Since $\sigma\circ \alpha =\Id_A$ and $\beta\circ \tau =\Id_C$, we have $\alpha$ is a closed immersion and $\beta$ is faithfully flat. Then we define a morphism $\varphi:B\rightarrow A\times_k C$ by $\varphi=(\sigma, \beta)$. Consider the following diagram:
    \[\begin{tikzcd}
        1\arrow[r]&A\arrow[r,shift left=2pt,harpoon,"\alpha"]\arrow[d,"\Id_A"] & B\arrow[l,shift left=2pt,harpoon,"\sigma"] \arrow[d,"\varphi"] \arrow[r,shift left=2pt,harpoon,"\beta"]&  C \arrow[l,shift left=2pt,harpoon,"\tau"]\arrow[r]\arrow[d,"\Id_C"]&1\\
        1\arrow[r]&A\arrow[r,shift left=2pt,harpoon,"i_A"]&A\times_k C \arrow[l,shift left=2pt,harpoon,"p_A"]\arrow[r,shift left=2pt,harpoon,"p_C"]&C\arrow[r]\arrow[l,shift left=2pt,harpoon,"i_C"]&1
    \end{tikzcd}\]
    where $i_A$ and $i_C$ are the natural inclusions, i.e. for any $k$-algebra $R$, $i_{A(R)}:A(R)\rightarrow (A\times_k C)(R)$ sends $a_R$ to $(a_R,e_{C(R)})$, $i_{C(R)}:C(R)\rightarrow (A\times_k C)(R)$ sends $c_R$ to $(e_{A(R)},c_R)$, $p_A$ and $p_C$ are the natural projections. For any $k$-algebra $R$, we have the commutative diagram of exact sequences of groups by \cite{Mil12},
    \[\begin{tikzcd}
        1\arrow[r]&A(R)\arrow[r,shift left=2pt,harpoon,"{\alpha}_R"]\arrow[d,"\Id_{A(R)}"] & B(R)\arrow[r,shift left=2pt,harpoon,"\beta_R"]\arrow[l,shift left=2pt,harpoon,"{\sigma_R}"] \arrow[d,"{\varphi}_R"] &  C(R) \arrow[l,shift left=2pt,harpoon,"\tau_R"]\arrow[d,"\Id_{C(R)}"]\arrow[r] &1\\
        1\arrow[r]&A(R)\arrow[r,shift left=2pt,harpoon,"i_{A(R)}"]&(A\times_k C)(R)\arrow[l,shift left=2pt,harpoon,"p_{A(R)}"]\arrow[r,shift left=2pt,harpoon,"p_{C(R)}"] &C(R)\arrow[l,shift left=2pt,harpoon,"i_{C(R)}"]\arrow[r]&1
    \end{tikzcd}\]
    where the surjection of $\beta_R:B(R)\rightarrow C(R)$ is preserved by $\beta_R\circ \tau_R=\Id_{C(R)}$.Then by diagram chasing, $\varphi_R: B(R)\rightarrow (A\times_k C)(R)$ is an isomorphism for any $k$-algebra $R$. Hence $\varphi:B\rightarrow A\times_k C$ is an isomorphism.
    
    $(2)\Rightarrow (1)$ Suppose there exists an isomorphism $\varphi:B\rightarrow A\times_k C$ with $\varphi\circ \alpha=i_A$ and $p_C\circ \varphi =\beta$. Obviously, the sequence $1\rightarrow A\xrightarrow{\alpha} B\xrightarrow{\beta} C\rightarrow 1$ is exact. Set $\sigma:=p_A\circ \varphi$ and $\tau:=\varphi^{-1}\circ i_C$. Then $\sigma\circ \alpha=p_A\circ \varphi \circ \alpha=p_A\circ i_A=\Id_A$ and $\beta\circ \tau=\beta\circ\varphi^{-1}\circ i_C=p_C\circ \varphi \circ \varphi^{-1}\circ i_C=p_C\circ i_C=\Id_C.$
\end{proof}

The following Theorem is the main result of this section.
\begin{Theorem}\label{kunnethMain}
    Let $k$ be a field, $X$ and $Y$ connected schemes proper over $k$, $x_0\in X(k)$, $y_0\in Y(k)$, $p_X: X\times_k Y\rightarrow X$ and $p_Y:X\times_k Y\rightarrow Y$ the projections, $i_X: X\rightarrow X\times_k Y$ send $x$ to $(x,y_0)$ and $i_Y: Y\rightarrow X\times_k Y$ send $y$ to $(x_0,y)$, $\mathcal{C}_X,\mathcal{C}_Y,\mathcal{C}_{X\times_k Y}$ the Tannakian categories over $X$, $Y$, $X\times_k Y$ respectively. Suppose pullback induce functors among Tannakian categories
$$\begin{aligned}
    \begin{tikzcd}
	Y\arrow[r,shift left=2pt,harpoon,"i_Y"]\arrow[d]&X\times_k Y\arrow[d,shift left=2pt,harpoon,"p_X"]\arrow[l,shift left=2pt,harpoon,"p_Y"]\\
	\Spec k\arrow[r]&X\arrow[u,shift left=2pt,harpoon,"i_X"]
    \end{tikzcd}&\quad\quad&\begin{tikzcd}
	\mathcal{C}_Y\arrow[r,shift left=2pt,harpoon,"p_Y^*"]&\mathcal{C}_{X\times_k Y}\arrow[d,shift left=2pt,harpoon,"i_X^*"]\arrow[l,shift left=2pt,harpoon,"i_Y^*"]\\
	&\mathcal{C}_X\arrow[u,shift left=2pt,harpoon,"p_X^*"]
    \end{tikzcd}\end{aligned}.$$
Then the following five conditions are equivalent:
\begin{enumerate}
    \item For any $E\in \mathcal{C}_{X\times_k Y}$, ${p_X}_*E\in\mathcal{C}_X$ and $E$ satisfies base change at $x_0$.
    \item For any $E\in \mathcal{C}_{X\times_k Y}$, ${p_Y}_*E\in\mathcal{C}_Y$ and $E$ satisfies base change at $y_0$.
    \item The homotopy sequence $1\rightarrow \pi(\mathcal{C}_Y,y_0)\rightarrow \pi(\mathcal{C}_{X\times_k Y},(x_0,y_0))\rightarrow \pi(\mathcal{C}_X,x_0)\rightarrow 1$ is exact.
    \item The homotopy sequence $1\rightarrow \pi(\mathcal{C}_X,x_0)\rightarrow \pi(\mathcal{C}_{X\times_k Y},(x_0,y_0))\rightarrow \pi(\mathcal{C}_Y,y_0)\rightarrow 1$ is exact.
    \item The natural homomorphism $\pi(\mathcal{C}_{X\times_k Y},(x_0,y_0))\rightarrow \pi(\mathcal{C}_X,x_0)\times_k \pi(\mathcal{C}_Y,y_0)$ is an isomorphism.
\end{enumerate}
\end{Theorem}
\begin{proof}
    $(1)\Rightarrow(3)$ For any $E\in \mathcal{C}_{X\times_k Y}$, we have  ${p_X}_*E\in\mathcal{C}_X$ and $E$ satisfies base change at $x_0$. Let $E'\in\mathcal{C}_{Y}$. Note that $\Id_Y=p_Y\circ i_Y$, so $E'\cong i_Y^*p_Y^* E'$. Then the functor $i_Y^*$ is observable by Lemma~\ref{observable}, and the homomorphism $\pi(\mathcal{C}_Y,y_0)\rightarrow \pi(\mathcal{C}_{X\times_k Y},(x_0,y_0))$ induced by $i_Y^*$ is a closed immersion by [Theorem~\ref{Thm1}, (2)]. Hence, by Theorem~\ref{Main}, we have an exact homotopy sequence
    $$1\rightarrow \pi(\mathcal{C}_Y,y_0)\rightarrow \pi(\mathcal{C}_{X\times_k Y},(x_0,y_0))\rightarrow \pi(\mathcal{C}_X,x_0)\rightarrow 1.$$

    $(3)\Rightarrow (5)$ Suppose we have an exact homotopy sequence 
    $$1\rightarrow \pi(\mathcal{C}_Y,y_0)\rightarrow \pi(\mathcal{C}_{X\times_k Y},(x_0,y_0))\rightarrow \pi(\mathcal{C}_X,x_0)\rightarrow 1.$$
    Since $p_Y\circ i_Y= \Id_Y$, so $i_Y^*\circ p_Y^*=\Id_{\mathcal{C}_Y}$. Then the homomorphism
    $$\pi(\mathcal{C}_Y,y_0)\rightarrow \pi(\mathcal{C}_{X\times_k Y},(x_0,y_0))\rightarrow \pi(\mathcal{C}_Y,y_0)$$
    corresponding to $i_Y^*\circ p_Y^*$ is $\Id_{\pi(\mathcal{C}_Y,y_0)}$. Similarly, the homomorphism
    $$\pi(\mathcal{C}_X,x_0)\rightarrow \pi(\mathcal{C}_{X\times_k Y},(x_0,y_0))\rightarrow \pi(\mathcal{C}_X,x_0)$$ 
    corresponding to $i_X^*\circ p_X^*$ is $\Id_{\pi(\mathcal{C}_X,x_0)}$. By Lemma~\ref{split}, we have a natural isomorphism $$\pi(\mathcal{C}_{X\times_k Y},(x_0,y_0))\xrightarrow{\cong} \pi(\mathcal{C}_X,x_0)\times_k \pi(\mathcal{C}_Y,y_0).$$

    $(5)\Rightarrow (1)$ Let $E\in\mathcal{C}_{X\times_k Y}$. By the isomorphism 
    $$\pi(\mathcal{C}_{X\times_k Y},(x_0,y_0))\xrightarrow{\cong} \pi(\mathcal{C}_X,x_0)\times_k \pi(\mathcal{C}_Y,y_0),$$ 
    there exists $E_X\in \mathcal{C}_X$ and $E_Y\in\mathcal{C}_Y$ such that $E\hookrightarrow p_X^*E_X\otimes_{\mathcal{O}_{X\times_k Y}}p_Y^*E_Y$. Consider the exact sequence
    $$0\rightarrow E\rightarrow p_X^*E_X\otimes_{\mathcal{O}_{X\times_k Y}}p_Y^*E_Y\rightarrow \coker\rightarrow 0.$$
    Since $\coker\in\mathcal{C}_{X\times_k Y}$, the same discussion applied to $\coker$, there exists $E'_X\in \mathcal{C}_X$ and $E'_Y\in\mathcal{C}_Y$ such that $\coker\hookrightarrow p_X^*E'_X\otimes_{\mathcal{O}_{X\times_k Y}}p_Y^*E'_Y$.
    Then we have an exact sequence in $\mathcal{C}_{X\times_k Y}$:
    $$0\rightarrow E\rightarrow p_X^*E_X\otimes_{\mathcal{O}_{X\times_k Y}}p_Y^*E_Y\rightarrow p_X^*E'_X\otimes_{\mathcal{O}_{X\times_k Y}}p_Y^*E'_Y.$$
    Applying the left exact functor ${p_X}_*$, and note that 
    $$\begin{aligned}
        {p_X}_*(p_X^*E_X\otimes_{\mathcal{O}_{X\times_k Y}}p_Y^*E_Y)&\cong E_X\otimes_{\mathcal{O}_X}{p_X}_*p_Y^*E_Y\cong E_X^{\oplus h^0(Y,E_Y)},\\
        {p_X}_*(p_X^*E'_X\otimes_{\mathcal{O}_{X\times_k Y}}p_Y^*E'_Y)&\cong E'_X\otimes_{\mathcal{O}_X}{p_X}_*p_Y^*E'_Y\cong {E'}_X^{\oplus h^0(Y,E'_Y)}.
    \end{aligned}$$
    Then we have the following exact sequence in $\mathcal{C}_X$:
    $$0\rightarrow {p_X}_*E\rightarrow E_X^{\oplus h^0(Y,E_Y)}\rightarrow {E'_X}^{\oplus h^0(Y,E'_Y)},$$
    which implies that $p_*E\in\mathcal{C}_X$. Applying the functor $|_{x_0}$ to the above exact sequence, we have the following commutative diagram
    \[
        \begin{tikzcd}
            0\arrow[r]&({p_X}_*E)|_{x_0}\arrow[r]\arrow[d]&(E_X|_{x_0})^{\oplus h^0(Y,E_Y)}\arrow[r]\arrow[d,"\cong"]&({{E'}_X}|_{x_0})^{\oplus h^0(Y,E'_Y)}\arrow[d,"\cong"]\\
            0\arrow[r]&H^0(Y,i_Y^*E)\arrow[r]&H^0(Y,E_X|_{x_0}\otimes_k E_Y)\arrow[r]&H^0(Y,E'_X|_{x_0}\otimes_k E'_Y)
        \end{tikzcd}
    \]
    It follows that $({p_X}_*E)|_{x_0}\cong H^0(Y,i_Y^*E)$, i.e. $E$ satisfies base change at $x_0$.
    
    By symmetry of $X$ and $Y$, the proof of $(2)\Rightarrow (4)\Rightarrow (5)\Rightarrow (2)$ is similar.
    \end{proof}

\begin{Proposition}\label{kunnethdescent}
    Let $k$ be a field, $X$ and $Y$ connected schemes proper over $k$, $x_0\in X(k)$, $y_0\in Y(k)$, $p_X: X\times_k Y\rightarrow X$ and $p_Y:X\times_k Y\rightarrow Y$ the projections, $i_X: X\rightarrow X\times_k Y$ send $x$ to $(x,y_0)$ and $i_Y: Y\rightarrow X\times_k Y$ send $y$ to $(x_0,y)$, $\mathcal{C}'_X,\mathcal{C}'_Y,\mathcal{C}'_{X\times_k Y}$ the Tannakian categories over $X$, $Y$, $X\times_k Y$ respectively, $\mathcal{C}_X\subseteq \mathcal{C}'_X,\mathcal{C}_Y\subseteq \mathcal{C}'_Y,\mathcal{C}_{X\times_k Y}\subseteq \mathcal{C}'_{X\times_k Y}$ the Tannakian subcategories. Suppose pullback induce functors among Tannakian categories
$$\begin{aligned}
    \begin{tikzcd}
	Y\arrow[r,shift left=2pt,harpoon,"i_Y"]\arrow[d]&X\times_k Y\arrow[d,shift left=2pt,harpoon,"p_X"]\arrow[l,shift left=2pt,harpoon,"p_Y"]\\
	\Spec k\arrow[r]&X\arrow[u,shift left=2pt,harpoon,"i_X"]
    \end{tikzcd}&\quad\quad&\begin{tikzcd}
	\mathcal{C}'_Y\arrow[r,shift left=2pt,harpoon,"p_Y^*"]&\mathcal{C}'_{X\times_k Y}\arrow[d,shift left=2pt,harpoon,"i_X^*"]\arrow[l,shift left=2pt,harpoon,"i_Y^*"]\\
	&\mathcal{C}'_X\arrow[u,shift left=2pt,harpoon,"p_X^*"]
    \end{tikzcd}&\quad\quad&\begin{tikzcd}
	\mathcal{C}_Y\arrow[r,shift left=2pt,harpoon,"p_Y^*"]&\mathcal{C}_{X\times_k Y}\arrow[d,shift left=2pt,harpoon,"i_X^*"]\arrow[l,shift left=2pt,harpoon,"i_Y^*"]\\
	&\mathcal{C}_X\arrow[u,shift left=2pt,harpoon,"p_X^*"]
    \end{tikzcd}\end{aligned}.$$
    If the natural homomorphism $\pi(\mathcal{C}'_{X\times_k Y},(x_0,y_0))\rightarrow \pi(\mathcal{C}'_X,x_0)\times_k \pi(\mathcal{C}'_Y,y_0)$ is an isomorphism, then the natural homomorphism $\pi(\mathcal{C}_{X\times_k Y},(x_0,y_0))\rightarrow \pi(\mathcal{C}_X,x_0)\times_k \pi(\mathcal{C}_Y,y_0)$ is an isomorphism.
\end{Proposition}

\begin{proof}
    Suppose the natural homomorphism $\pi(\mathcal{C}'_{X\times_k Y},(x_0,y_0))\rightarrow \pi(\mathcal{C}'_X,x_0)\times_k \pi(\mathcal{C}'_Y,y_0)$ is an isomorphism. By Theorem~\ref{kunnethMain}, the natural homomorphism $\pi(\mathcal{C}'_{X\times_k Y},(x_0,y_0))\twoheadrightarrow \pi(\mathcal{C}'_X,x_0)$ is faithfully flat and for any $E'\in\mathcal{C}'_{X\times_k Y}$, we have ${p_X}_*E'\in\mathcal{C}'_X$ and $E'$ satisfies base change at $x_0$. Then by Proposition~\ref{embedding}, the canonical map $p_X^*{p_X}_*E'\hookrightarrow E'$ is an embedding in $\mathcal{C}'_{X\times_k Y}$ for any $E'\in\mathcal{C}'_{X\times_k Y}$. 
    
    Let $E\in\mathcal{C}_{X\times_k Y}\subseteq \mathcal{C}'_{X\times_k Y}$, then $p_X^*{p_X}_*E\hookrightarrow E\in\mathcal{C}_{X\times_k Y}$. It follows that $${p_X}_*E\cong i_X^*p_X^*{p_X}_*E\in\mathcal{C}_X.$$
    Then by Theorem~\ref{kunnethMain}, we have a natural isomorphism
    $$\pi(\mathcal{C}_{X\times_k Y},(x_0,y_0))\rightarrow \pi(\mathcal{C}_X,x_0)\times_k \pi(\mathcal{C}_Y,y_0).$$
\end{proof}

\begin{Proposition}\label{basechangeproduct}
Let $k$ be a field, $X$ and $Y$ connected schemes proper over $k$, $x_0\in X(k)$, $y_0\in Y(k)$, $p_X: X\times_k Y\rightarrow X$ and $p_Y:X\times_k Y\rightarrow Y$ the projections, $i_X: X\rightarrow X\times_k Y$ send $x$ to $(x,y_0)$ and $i_Y: Y\rightarrow X\times_k Y$ send $y$ to $(x_0,y)$, $\mathcal{C}_X,\mathcal{C}_Y,\mathcal{C}_{X\times_k Y}$ the Tannakian categories over $X$, $Y$, $X\times_k Y$ respectively. Suppose pullback induce functors among Tannakian categories
$$\begin{aligned}
    \begin{tikzcd}
	Y\arrow[r,shift left=2pt,harpoon,"i_Y"]\arrow[d]&X\times_k Y\arrow[d,shift left=2pt,harpoon,"p_X"]\arrow[l,shift left=2pt,harpoon,"p_Y"]\\
	\Spec k\arrow[r]&X\arrow[u,shift left=2pt,harpoon,"i_X"]
    \end{tikzcd}&\quad\quad&\begin{tikzcd}
	\mathcal{C}_Y\arrow[r,shift left=2pt,harpoon,"p_Y^*"]&\mathcal{C}_{X\times_k Y}\arrow[d,shift left=2pt,harpoon,"i_X^*"]\arrow[l,shift left=2pt,harpoon,"i_Y^*"]\\
	&\mathcal{C}_X\arrow[u,shift left=2pt,harpoon,"p_X^*"]
    \end{tikzcd}\end{aligned}.$$
If for any $E\in\mathcal{C}_{X\times_k Y}$, we have ${p_X}_*E\in\mathcal{C}_X$ and $E$ satisfies base change at $x_0$. Then 
\begin{enumerate}
    \item For any $E\in\mathcal{C}_{X\times_k Y}$, we have $R^i{p_X}_*E\in\mathcal{C}_X$ for any $i\geq 0$.
    \item For any $E\in\mathcal{C}_{X\times_k Y}$, we have $R^{i}{p_X}_*E|_{x_0}\cong H^{i}(Y,i_Y^*E)$ for any $i\geq 0$.
\end{enumerate}

\end{Proposition}

\begin{proof}
Let $E\in\mathcal{C}_{{X\times_k Y}}$, we have $\mathcal{H}om_{X\times_k Y}(p_Y^*i_Y^* E, E)\in\mathcal{C}_{X\times_k Y}$. Then ${p_X}_*\mathcal{H}om_{X\times_k Y}(p_Y^*i_Y^* E, E)\in\mathcal{C}_X$. Tensoring the adjunction 
$$p_X^*{p_X}_* \mathcal{H}om_{X\times_k Y}(p_Y^*i_Y^* E, E)\rightarrow \mathcal{H}om_{X\times_k Y}(p_Y^*i_Y^* E, E)$$ 
with $p_Y^*i_Y^* E$ and compositing with 
$$\begin{aligned}
    \mathcal{H}om_{X\times_k Y}(p_Y^*i_Y^* E, E)\otimes_{\mathcal{O}_{X\times_k Y}} p_Y^*i_Y^* E\rightarrow E,\varphi\otimes e\mapsto  \varphi(e),
\end{aligned}$$
we have a morphism in $\mathcal{C}_{X\times_k Y}$:
$$p_X^*{p_X}_* \mathcal{H}om_{X\times_k Y}(p_Y^*i_Y^* E, E)\otimes_{\mathcal{O}_{X\times_k Y}} p_Y^*i_Y^* E\rightarrow \mathcal{H}om_{X\times_k Y}(p_Y^*i_Y^* E, E)\otimes_{\mathcal{O}_{X\times_k Y}} p_Y^*i_Y^* E\rightarrow E.$$ 
Since $\mathcal{H}om_{X\times_k Y}(p_Y^*i_Y^* E, E)$ satisfies base change at $x_0$, we have
$$\begin{aligned}
    ({p_X}_* \mathcal{H}om_{X\times_k Y}(p_Y^*i_Y^* E, E))|_{x_0}&\cong H^0(Y,i_Y^*\mathcal{H}om_{X\times_k Y}(p_Y^*i_Y^* E, E))\\
    &\cong H^0(Y,\mathcal{H}om_{X\times_k Y}(i_Y^*p_Y^*i_Y^* E, i_Y^*E))\\
    &\cong \Hom_{Y}(i_Y^*E,i_Y^*E).
\end{aligned}$$
Restricting the above morphism to $\{x_0\}\times_k Y$, we obtain a surjective morphism
$$\begin{aligned}
    i_Y^*p_X^*{p_X}_* \mathcal{H}om_{X\times_k Y}(p_Y^*i_Y^* E, E)\otimes_{\mathcal{O}_{X\times_k Y}} i_Y^*p_Y^*i_Y^* E&\cong ({p_X}_* \mathcal{H}om_{X\times_k Y}(p_Y^*i_Y^* E, E))|_{x_0}\otimes_{k} i_Y^* E\\
    &\cong \Hom_{Y}(i_Y^*E,i_Y^*E)\otimes_{k}i_Y^*E\twoheadrightarrow i_Y^*E.
\end{aligned}$$
It follows that we have a surjective morphism 
$$p_X^*{p_X}_* \mathcal{H}om_{X\times_k Y}(p_Y^*i_Y^* E, E)\otimes_{\mathcal{O}_{X\times_k Y}} p_Y^*i_Y^* E\twoheadrightarrow E.$$
So for any $E\in\mathcal{C}_{X\times_k Y}$, there exist $E'\in\mathcal{C}_X,F'\in\mathcal{C}_Y$ such that we have a quotient in $\mathcal{C}_{X\times_k Y}$:
$$p_X^* E' \otimes p_Y^*F'\twoheadrightarrow E.$$
Then by duality, there exist $E_0\in\mathcal{C}_X,F_0\in\mathcal{C}_Y$ such that we have an embedding in $\mathcal{C}_{X\times_k Y}$:
$$E\hookrightarrow p_X^* E_0 \otimes p_Y^*F_0.$$
Since $(p_X^* E_0 \otimes p_Y^*F_0)/E\in\mathcal{C}_{X\times_k Y}$, there exist $E_1\in\mathcal{C}_X,F_1\in\mathcal{C}_Y$ such that we have an embedding in $\mathcal{C}_{X\times_k Y}$:
$$(p_X^* E_0 \otimes p_Y^*F_0)/E\hookrightarrow p_X^* E_1 \otimes p_Y^*F_1.$$
Inductively we can therefore construct the following acyclic complex of objects in $\mathcal{C}_{X\times_k Y}$:
$$\mathcal{C}^{\bullet}:
0\rightarrow E\rightarrow p_X^* E_0 \otimes p_Y^*F_0\rightarrow p_X^* E_1 \otimes p_Y^*F_1\rightarrow\cdots p_X^* E_i \otimes p_Y^*F_i\rightarrow\cdots,$$
where $E_i\in\mathcal{C}_X$ and $F_i\in\mathcal{C}_Y$ for any $i$. Note that by projection formula, we have
$$R^i{p_X}_*(p_X^* E_j \otimes p_Y^*F_j)\cong E_j\otimes R^i{p_X}_*(p_Y^*F_j) \cong E_j^{\oplus h^i(Y,F_j)}\in \mathcal{C}_X.$$
Consider the following spectral sequence
$$E_j^{\oplus h^i(Y,F_j)}=R^i{p_X}_*(p_X^* E_j \otimes p_Y^*F_j)\Rightarrow R^{i+j}{p_X}_*\mathcal{C}^{\bullet}\cong R^{i+j}{p_X}_*E.$$
Since $\mathcal{C}_X$ is an abelian category, kernels and cokernels of morphisms of objects in this category are also in $\mathcal{C}_X$. This implies that $R^{i+j}{p_X}_* E\in\mathcal{C}_X$ for any $i+j\geq 0$. 

Note that the above complex restricted to $\{x_0\}\times_k Y$ remains acyclic, as all the sheaves in this complex are locally free. Therefore, we have a commutative diagram of spectral sequences
$$\begin{tikzcd}
     R^i{p_X}_*(p_X^* E_j \otimes p_Y^*F_j)|_{x_0}\arrow{d}&\Rightarrow & R^{i+j}{p_X}_*E|_{x_0}\arrow{d}\\
     H^i(Y,i_Y^*(p_X^* E_j \otimes p_Y^*F_j))&\Rightarrow&H^{i+j}(Y,i_Y^*E)
\end{tikzcd}$$
The left vertical map in this diagram is an isomorohism. Therefore, the right vertical map is an isomorphism:
$$R^{l}{p_X}_*E|_{x_0}\cong H^{l}(Y,i_Y^*E)$$
for any $l\geq 0$.
\end{proof}
\begin{Lemma}\label{saturatedinducefunctor}
    Let $k$ be a field, $f:X\rightarrow S$ a morphism of connected schemes proper over $k$, $x\in X(k)$ lying over $s\in S(k)$, $\mathcal{C}_X,\mathcal{C}_S$ the Tannakian categories over $X$, $S$ respectively. If pullback induces functor $f^*:\mathcal{C}_S\rightarrow \mathcal{C}_X$, then pullback induces functor $f^*:\overline{\mathcal{C}}_S\rightarrow \overline{\mathcal{C}}_X$.
\end{Lemma}

\begin{proof}
    For any $E\in\overline{\mathcal{C}}_S$, consider the filtration of $E$ in $\overline{\mathcal{C}}_S$:
    $$0\hookrightarrow E_1\hookrightarrow \cdots\hookrightarrow E_n=E,$$
    such that $E_{i+1}/E_i\in\mathcal{C}_S$ for any $i$. Applying $f^*$, we obtain a filtration of $f^* E$:
    $$0\hookrightarrow f^*E_1\hookrightarrow \cdots\hookrightarrow f^*E_n=f^*E,$$
    where $(f^*E_{i+1})/(f^*E_i)\cong f^* (E_{i+1}/E_i)\in \mathcal{C}_X$ for any $i$. Hence $f^*E\in\overline{\mathcal{C}}_X$, i.e. pullback induces functor $$f^*:\overline{\mathcal{C}}_S\rightarrow \overline{\mathcal{C}}_X.$$
\end{proof}

\begin{Proposition}\label{SaturatedKunneth}
    Let $k$ be a field, $X$ and $Y$ connected schemes proper over $k$, $x_0\in X(k)$, $y_0\in Y(k)$, $p_X: X\times_k Y\rightarrow X$ and $p_Y:X\times_k Y\rightarrow Y$ the projections, $i_X: X\rightarrow X\times_k Y$ send $x$ to $(x,y_0)$ and $i_Y: Y\rightarrow X\times_k Y$ send $y$ to $(x_0,y)$, $\mathcal{C}_X,\mathcal{C}_Y,\mathcal{C}_{X\times_k Y}$ the Tannakian categories over $X$, $Y$, $X\times_k Y$ respectively. Suppose pullback induce functors among Tannakian categories
$$\begin{aligned}
    \begin{tikzcd}
	Y\arrow[r,shift left=2pt,harpoon,"i_Y"]\arrow[d]&X\times_k Y\arrow[d,shift left=2pt,harpoon,"p_X"]\arrow[l,shift left=2pt,harpoon,"p_Y"]\\
	\Spec k\arrow[r]&X\arrow[u,shift left=2pt,harpoon,"i_X"]
    \end{tikzcd}&\quad\quad&\begin{tikzcd}
	\mathcal{C}_Y\arrow[r,shift left=2pt,harpoon,"p_Y^*"]&\mathcal{C}_{X\times_k Y}\arrow[d,shift left=2pt,harpoon,"i_X^*"]\arrow[l,shift left=2pt,harpoon,"i_Y^*"]\\
	&\mathcal{C}_X\arrow[u,shift left=2pt,harpoon,"p_X^*"]
    \end{tikzcd}\end{aligned}.$$
Then the following conditions are equivalent:
\begin{enumerate}
    \item The natural homomorphism $\pi(\mathcal{C}_{X\times_k Y},(x_0,y_0))\rightarrow \pi(\mathcal{C}_X,x_0)\times_k \pi(\mathcal{C}_Y,y_0)$ is an isomorphism.
    \item The natural homomorphism $\pi(\overline{\mathcal{C}}_{X\times_k Y},(x_0,y_0))\rightarrow \pi(\overline{\mathcal{C}}_X,x_0)\times_k \pi(\overline{\mathcal{C}}_Y,y_0)$ is an isomorphism.
\end{enumerate}
\end{Proposition}

\begin{proof}
    $(2)\Rightarrow (1)$ It follows by Proposition~\ref{kunnethdescent}.
    
    $(1)\Rightarrow (2)$ Since the natural homomorphism $\pi(\mathcal{C}_{X\times_k Y},(x_0,y_0))\rightarrow \pi(\mathcal{C}_X,x_0)\times_k \pi(\mathcal{C}_Y,y_0)$ is an isomorphism, for any $F\in\mathcal{C}_{X\times_k Y}$, we have ${p_X}_*E\in\mathcal{C}_X$ and $E$ satisfies base change at $x_0$ by Theorem~\ref{kunnethMain}.
    By Lemma~\ref{saturatedinducefunctor}, pullback induce functors among Tannakian categories
    \[
        \begin{tikzcd}
	\overline{\mathcal{C}}_Y\arrow[r,shift left=2pt,harpoon,"p_Y^*"]&\overline{\mathcal{C}}_{X\times_k Y}\arrow[d,shift left=2pt,harpoon,"i_X^*"]\arrow[l,shift left=2pt,harpoon,"i_Y^*"]\\
	&\overline{\mathcal{C}}_X\arrow[u,shift left=2pt,harpoon,"p_X^*"]
    \end{tikzcd}.
    \]
    
    For any $E\in\overline{\mathcal{C}}_{X\times_k Y}$, consider the filtration of $E$ in $\overline{\mathcal{C}}_{X\times_k Y}$:
    $$0\hookrightarrow E_1\hookrightarrow \cdots\hookrightarrow E_n=E,$$
    such that $E_{i+1}/E_i\in\mathcal{C}_{X\times_k Y}$ for any $i$. Applying the functor ${p_X}_*$, we obtain a filtration of ${p_X}_*E$:
$$0\hookrightarrow {p_X}_*E_1\hookrightarrow \cdots\hookrightarrow {p_X}_*E_n={p_X}_*E.$$

(I) We firstly verify that ${p_X}_*E\in\mathcal{C}_X$. Consider the exact sequence
$$0\rightarrow E_1\rightarrow E_2\rightarrow E_2/E_1\rightarrow 0.$$
Applying ${p_X}_*$, we obtain an exact sequence
$$\begin{aligned}
    0\xrightarrow{\alpha_0} {p_X}_*E_1\rightarrow {p_X}_*E_2\rightarrow {p_X}_*(E_2/E_1)\xrightarrow{\alpha_1} R^1{p_X}_*E_1\rightarrow R^1{p_X}_*E_2\rightarrow R^1{p_X}_*(E_2/E_1)\xrightarrow{\alpha_2}
    R^2{p_X}_*E_1\rightarrow\\
    \cdots\rightarrow R^{n-1}{p_X}_*(E_2/E_1)\xrightarrow{\alpha_{n}}R^n{p_X}_*E_1\rightarrow R^n{p_X}_*E_2\rightarrow R^n{p_X}_*(E_2/E_1)\xrightarrow{\alpha_{n+1}}R^{n+1}{p_X}_*E_1\rightarrow\cdots
\end{aligned}$$
Then we have short exact sequences for any $i$:
$$0\rightarrow \coker \alpha_i\rightarrow R^i{p_X}_*E_2\rightarrow \ker \alpha_{i+1} \rightarrow 0.$$
Then $R^i{p_X}_*E_1,R^i{p_X}_*(E_2/E_1)\in\mathcal{C}_X$ for any $i\geq 0$ by Proposition~\ref{basechangeproduct}. So $\ker\alpha_i,\coker\alpha_i\in\mathcal{C}_X$ for any $i\geq 0$. Then we have $R^i{p_X}_*E_2\in\overline{\mathcal{C}}_X$ for any $i\geq 0$.

Consider the exact sequence
$$0\rightarrow E_2\rightarrow E_3\rightarrow E_3/E_2\rightarrow 0.$$
Applying ${p_X}_*$, we obtain an exact sequence
$$\begin{aligned}
    0\xrightarrow{\beta_0} {p_X}_*E_2\rightarrow {p_X}_*E_3\rightarrow {p_X}_*(E_3/E_2)\xrightarrow{\beta_1} R^1{p_X}_*E_2\rightarrow R^1{p_X}_*E_3\rightarrow R^1{p_X}_*(E_3/E_2)\xrightarrow{\beta_2}
    R^2{p_X}_*E_2\rightarrow\\
    \cdots\rightarrow R^{n-1}{p_X}_*(E_3/E_2)\xrightarrow{\beta_{n}}R^n{p_X}_*E_2\rightarrow R^n{p_X}_*E_3\rightarrow R^n{p_X}_*(E_3/E_2)\xrightarrow{\beta_{n+1}}R^{n+1}{p_X}_*E_2\rightarrow\cdots
\end{aligned}$$
Then we have short exact sequences for any $i$:
$$0\rightarrow \coker \beta_i\rightarrow R^i{p_X}_*E_3\rightarrow \ker \beta_{i+1} \rightarrow 0.$$
Then $R^i{p_X}_*E_2,R^i{p_X}_*(E_3/E_2)\in\mathcal{C}_X$ for any $i\geq 0$ by Proposition~\ref{basechangeproduct}. So $\ker\beta_i,\coker\beta_i\in\mathcal{C}_X$ for any $i\geq 0$. Then we have $R^i{p_X}_*E_3\in\overline{\mathcal{C}}_X$ for any $i\geq 0$.

Repeating this method, we have $R^i {p_X}_*E_j\in\overline{\mathcal{C}}_X$ for any $i\geq 0$ and any $j\geq 1$. In particular, we have $${p_X}_*E\in\overline{\mathcal{C}}_X.$$

(II) We verify that $E$ satisfies base change at $x_0$. Consider the following commutative diagram
\[
    \begin{tikzcd}
        0\arrow[r]& ({p_X}_*E_1)|_{x_0}\arrow[r]\arrow[d,"\cong"]&({p_X}_*E_2)|_{x_0}\arrow[r]\arrow[d]&({p_X}_*(E_2/E_1))|_{x_0}\arrow[d,"\cong"]\\
        0\arrow[r]& H^0(Y,i_Y^*E_1)\arrow[r]&H^0(Y,i_Y^*E_2)\arrow[r]&H^0(Y,i_Y^*(E_2/E_1))
    \end{tikzcd}
\]
It follows that $({p_X}_*E_2)|_{x_0}\cong H^0(Y,i_Y^*E_2)$. Repeating this method, we have $({p_X}_*E_i)|_{x_0}\cong H^0(Y,i_Y^*E_i)$ for any $i\geq 1$. In particular, we have 
$$({p_X}_*E)|_{x_0}\cong H^0(Y,i_Y^*E).$$
Then by Theorem~\ref{kunnethMain}, we have an isomorphism
$$\pi(\overline{\mathcal{C}}_{X\times_k Y},(x_0,y_0))\cong \pi(\overline{\mathcal{C}}_X,x_0)\times_k \pi(\overline{\mathcal{C}}_Y,y_0).$$
\end{proof}

\section{Applications}
\begin{Definition}
Let $k$ be a field, $X$ a geometrically reduced connected scheme proper over $k$, $x\in X(k)$, $E$ a vector bundle on $X$ of rank $r$. If $k$ is of positive characteristic, then $F_X:X\rightarrow X$ is the absolute Frobenius morphism. Then $E$ is said to be
\begin{itemize}
    \item \textit{numerically flat}, if both $E$ and $E^\vee$ are nef.
    \item \textit{Nori semistable}, if for any smooth proper curve $f:C\rightarrow X$, $f^*E$ is semistable of degree 0 on $C$.
    \item \textit{finite}, if there exist $f(t)\neq g(t)\in\mathbb{N}[t]$ such that $f(E)\cong g(E)$, where
    $$h(E):=\bigoplus_{i=0}^m\bigoplus_{j=1}^{n_i}E^{\otimes i}\text{ for any }h(t)=\sum\limits_{i=0}^mn_it^i\in\mathbb{N}[t].$$
    \item \textit{essentially finite}, if there exists $E_1\hookrightarrow E_2\hookrightarrow F\in\Vect(X)$ such that $E\cong E_2/E_1$, where $E_1,E_2$ are numerically flat and $F$ is finite.
    \item \textit{Frobenius finite}, if there exist $f(t)\neq g(t)\in\mathbb{N}[t]$ such that $\tilde{f}(E)\cong \tilde{g}(E)$, where 
    $$\tilde{h}(E):=\bigoplus\limits_{i=1}^{m}((F_X^i)^*E)^{\oplus n_i} \text{ for any }h(t)=\sum\limits_{i=0}^mn_it^i\in\mathbb{N}[t].$$
    Denote the set of all Frobenius finite vector bundles on $X$ by $FF(X)$.
    \item \textit{essentially Frobenius finite}, if there exists $E_1\hookrightarrow E_2\hookrightarrow F\in\Vect(X)$ such that $E\cong E_2/E_1$, where $E_1,E_2$ are numerically flat and $F\in FF(X)$.
    \item \textit{Frobenius trivial}, if there exists a positive integer $n$ such that $F_X^{n*} E\cong \mathcal{O}_X^{\oplus r}$.
    \item \textit{\'etale trivializable}, if there exists a finite \'etale covering $\phi: P\rightarrow X$ such that $\phi^* E\cong \mathcal{O}_P^{\oplus r}$.
    \item \textit{unipotent}, if there exists a filtration
    $0\hookrightarrow E_1\hookrightarrow \cdots\hookrightarrow E_n=E$
    such that $E_{i+1}/E_i\cong\mathcal{O}_X$ for any $i$.
\end{itemize}
We have the following Tannakian categories:
    \begin{itemize}
        \item $\mathcal{C}^{NF}(X)$: objects consist of numerically flat bundles on $X$.
        \item $\mathcal{C}^{N}(X)$: objects consist of essentially finite bundles on $X$.
        \item $\mathcal{C}^{F}(X)$: objects consist of essentially Frobenius finite bundles on $X$.
        \item $\mathcal{C}^{Loc}(X)$: objects consist of Frobenius trivial bundles on $X$.
        \item $\mathcal{C}^{\acute{e}t}(X)$: objects consist of \'etale trivializable bundles on $X$.
        \item $\mathcal{C}^{uni}(X)$: objects consist of unipotent bundles on $X$.
    \end{itemize}
We have the following Tannaka group schemes:
\begin{itemize}
        \item $\pi^{S}(X,x):=\pi(\mathcal{C}^{NF}(X),x)$, called the \textit{S-fundamental group scheme}.
        \item $\pi^{N}(X,x):=\pi(\mathcal{C}^{N}(X),x)$, called the \textit{Nori fundamental group scheme}.
        \item $\pi^{EN}(X,x):=\pi(\overline{\mathcal{C}^{N}(X)},x)$, called the \textit{extended Nori fundamental group scheme}.
        \item $\pi^{F}(X,x):=\pi(\mathcal{C}^{F}(X),x)$, called the \textit{F-fundamental group scheme}.
        \item $\pi^{EF}(X,x):=\pi(\overline{\mathcal{C}^{F}(X)},x)$, called the \textit{extended F-fundamental group scheme}.
        \item $\pi^{Loc}(X,x):=\pi(\mathcal{C}^{Loc}(X),x)$, called the \textit{local fundamental group scheme}.
        \item $\pi^{ELoc}(X,x):=\pi(\overline{\mathcal{C}^{Loc}(X)},x)$, called the \textit{extended local fundamental group scheme}.
        \item $\pi^{\acute{e}t}(X,x):=\pi(\mathcal{C}^{\acute{e}t}(X),x)$, called the \textit{\'etale fundamental group scheme}.
        \item $\pi^{E\acute{e}t}(X,x):=\pi(\overline{\mathcal{C}^{\acute{e}t}(X)},x)$, called the \textit{extended \'etale fundamental group scheme}.
        \item $\pi^{uni}(X,x):=\pi(\mathcal{C}^{uni}(X),x)$, called the \textit{unipotent fundamental group scheme}.
    \end{itemize}
\end{Definition}

\begin{Remark}
    Let $k$ be a field, $X$ a geometrically reduced connected scheme proper over $k$. Then 
    \begin{enumerate}
        \item $\mathcal{C}^{NF}(X)$ and $\mathcal{C}^{uni}(X)$ are saturated categories, i.e. $\overline{\mathcal{C}^{NF}(X)}=\mathcal{C}^{NF}(X)$ and $\overline{\mathcal{C}^{uni}(X)}=\mathcal{C}^{uni}(X)$.
        \item $\mathcal{C}^{*}(X)$ and $\overline{\mathcal{C}^{*}(X)}$ are Tannakian subcategories of $\mathcal{C}^{NF}(X)$ for $*\in\{N,F,Loc,\acute{e}t,uni\}$.
    \end{enumerate}
\end{Remark}

\begin{Remark}\label{tannakianrelation}
    Let $k$ be a field, $X$ a geometrically reduced connected scheme proper over $k$. Then a vector bundle $E$ is Nori semistable iff $E$ is numerically flat. Moreover, by \cite{AdAm25}, \cite{AmBi10}, \cite{Lan11}, \cite{MeSu08}, \cite{Nor76}, \cite{Nor82} and \cite{Ota17}, we have sequences of Tannakian subcategories
    $$\mathcal{C}^{\acute{e}t}(X)\subseteq\mathcal{C}^{N}(X)\subseteq\overline{\mathcal{C}^{N}(X)}\subseteq\mathcal{C}^{S}(X),\quad\mathcal{C}^{uni}(X)\subseteq\overline{\mathcal{C}^{\acute{e}t}(X)}\subseteq\overline{\mathcal{C}^{N}(X)}\subseteq\mathcal{C}^{S}(X).$$
    Let $k$ be a field of characteristic $p>0$. Then we have sequences of Tannakian subcategories
    $$\mathcal{C}^{Loc}(X)\subseteq\mathcal{C}^{N}(X)\subseteq\mathcal{C}^{EN}(X)\subseteq\mathcal{C}^{S}(X),\quad\mathcal{C}^{uni}(X)\subseteq\mathcal{C}^{ELoc}(X)\subseteq\mathcal{C}^{EN}(X)\subseteq\mathcal{C}^{S}(X).$$
\end{Remark}

\begin{Lemma}[{\cite[Lemma~1.4.3]{CLM22}}]\label{pullbacknef}
    Let $k$ be a field, $X$ a proper algebraic space over $k$, $f : Y\rightarrow X$ a proper morphism of algebraic spaces, $E$ a finite locally free $\mathcal{O}_X$-module. Then
    \begin{enumerate}
        \item If $E$ is nef, then $f^*E$ is nef.
        \item If $f$ is surjective and $f^*E$ is nef, then $E$ is nef.
    \end{enumerate}
\end{Lemma}

The following results are well-known. We prove them for convenience.

\begin{Proposition}\label{pullbackall}
    Let $k$ be a field, $f:X\rightarrow S$ a morphism of geometrically reduced connected schemes proper over $k$. Then
    \begin{enumerate}
        \item For any $E\in\mathcal{C}^{NF}(S)$, we have $f^*E\in\mathcal{C}^{NF}(X)$.
        \item For any $E\in\mathcal{C}^{N}(S)$, we have $f^*E\in\mathcal{C}^{N}(X)$.
        \item For any $E\in\mathcal{C}^{F}(S)$, we have $f^*E\in\mathcal{C}^{F}(X)$.
        \item For any $E\in\mathcal{C}^{Loc}(S)$, we have $f^*E\in\mathcal{C}^{Loc}(X)$.
        \item For any $E\in\mathcal{C}^{\acute{e}t}(S)$, we have $f^*E\in\mathcal{C}^{\acute{e}t}(X)$.
        \item For any $E\in\mathcal{C}^{uni}(S)$, we have $f^*E\in\mathcal{C}^{uni}(X)$.
    \end{enumerate}
\end{Proposition}

\begin{proof}
    (1) It follows by Lemma~\ref{pullbacknef}.

    (2) Let $F$ be a finite vector bundle on $S$, then there exist distinct $g(t),h(t)\in\mathbb{N}[t]$ such that $g(F)\cong h(F)$. It follows that $$g(f^*F)\cong f^*g(F)\cong f^*h(F)\cong h(f^*F),$$
    i.e. $f^*F$ is a finite vector bundle on $X$.

    If $E\in\mathcal{C}^N(S)$, then $E\cong F_1/F_2$, where $F_2\hookrightarrow F_1\hookrightarrow F$, $F_1,F_2\in\mathcal{C}^{NF}(S)$ and $F$ is a finite bundle on $S$. Then we have $f^*F_2\hookrightarrow f^*F_2\hookrightarrow f^*F$, where $f^*F_1,f^*F_2\in\mathcal{C}^{NF}(X)$ by (1). Hence $f^*E\cong f^*F_1/f^*F_2\in\mathcal{C}^N(X)$.

    (3) Let $F$ be a Frobenius finite vector bundle on $S$, then there exist distinct $g(t),h(t)\in\mathbb{N}[t]$ such that $\tilde{g}(F)\cong \tilde{h}(F)$. It follows that $$\tilde{g}(f^*F)\cong f^*\tilde{g}(F)\cong f^*\tilde{h}(F)\cong \tilde{h}(f^*F),$$
    i.e. $f^*F$ is a Frobenius finite vector bundle on $X$.

    If $E\in\mathcal{C}^F(S)$, then $E\cong F_1/F_2$, where $F_2\hookrightarrow F_1\hookrightarrow F$, $F_1,F_2\in\mathcal{C}^{NF}(S)$ and $F$ is a Frobenius finite vector bundle on $S$. Then we have $f^*F_2\hookrightarrow f^*F_2\hookrightarrow f^*F$, where $f^*F_1,f^*F_2\in\mathcal{C}^{NF}(X)$ by (1). Hence $$f^*E\cong f^*F_1/f^*F_2\in\mathcal{C}^F(X).$$

    (4) Let $E\in\mathcal{C}^{Loc}(S)$, then there exists $n$ such that $F_S^{n*}E\cong \mathcal{O}_S^{\oplus r}$. It follows that $$F_{X}^{n*}f^*E\cong f^* F_S^{n*}E\cong f^*\mathcal{O}_S^{\oplus r}\cong \mathcal{O}_X^{\oplus r}\text{, i.e. }f^*E\in\mathcal{C}^{Loc}(X).$$

    (5) Let $E\in\mathcal{C}^{\acute{e}t}(S)$, then there exists a finite \'etale covering $\phi:P\rightarrow S$ such that $\phi^*E\cong \mathcal{O}_P^{\oplus r}$. Consider the pullback diagram
    \[
        \begin{tikzcd}
            f^*P\arrow[r,"\phi'"]\arrow[d,"f'"]\arrow[dr, phantom, "\ulcorner", very near start]&X\arrow[d,"f"]\\
            P\arrow[r,"\phi"]&S
        \end{tikzcd}
    \]
    Then $\phi':f^*P\rightarrow X$ is a finite \'etale covering and
    $${\phi'}^{*}f^*E\cong {f'}^*\phi^*E\cong {f'}^*\mathcal{O}_P^{\oplus r}\cong \mathcal{O}_{f^*P}^{\oplus r}\text{, i.e. }f^*E\in\mathcal{C}^{\acute{e}t}(X).$$

    (6) Let $E\in\mathcal{C}^{uni}(S)$, then there exists a filtration
    $$0\hookrightarrow E_1\hookrightarrow \cdots\hookrightarrow E_n=E,$$
    such that $E_{i+1}/E_i\cong\mathcal{O}_S$ for any $i$. Applying the functor $f^*$, we obtain a filtration of $f^*E$:
    $$0\hookrightarrow f^*E_1\hookrightarrow \cdots\hookrightarrow f^*E_n=f^*E,$$
    such that $f^*E_{i+1}/f^*E_i\cong f^*(E_{i+1}/E_i)\cong \mathcal{O}_X$ for any $i$. Hence $f^*E\in\mathcal{C}^{uni}(X)$.
\end{proof}

\begin{Proposition}\label{pullbacksaturated}
    Let $k$ be a field, $f:X\rightarrow S$ a morphism of geometrically reduced connected schemes proper over $k$, $*\in\{N,F,Loc,\acute{e}t\}$. Then for any $E\in\overline{\mathcal{C}^{*}(S)}$, we have $f^*E\in\overline{\mathcal{C}^{*}(X)}$.
\end{Proposition}

\begin{proof}
    It follows by Lemma~\ref{saturatedinducefunctor} and Proposition~\ref{pullbackall}.
\end{proof}

\begin{Lemma}[{\cite[Theorem~4.1]{Lan12}}]\label{Slangerkunneth}
    Let $k$ be an algebraically closed field, $X$ and $Y$ reduced connected schemes proper over $k$, $x\in X(k)$, $y\in Y(k)$. Then the canonical map 
$$\pi^S(X\times_k Y,(x,y))\rightarrow \pi^S(X,x)\times_k \pi^S(Y,y)$$
is an isomorphism of $k$-group schemes.
\end{Lemma}

\begin{Lemma}[{\cite[Lemma~1.4.4]{CLM22}}]\label{basefieldnef}
    Let $X$ be a proper algebraic space over $k$, $E\in\Vect(X)$. Then $E$ is nef iff for every field extension $k\subseteq k'$, the pullback $E\otimes_k k'$ on $X_{k'}$ is nef.
\end{Lemma}

\begin{Lemma}\label{NFiff}
    Let $k$ be a field, $K/k$ a field extension, $X$ a connected scheme proper over $k$, $E\in\Vect(X)$. Then $E\in\mathcal{C}^{NF}(X)$ iff $E\otimes_k K\in \mathcal{C}^{NF}(X_K)$.
\end{Lemma}

\begin{proof}
    If follows immediately by Lemma~\ref{basefieldnef}.
\end{proof}

The K\"unneth formula for certain fundamental group schemes has been studied by many mathematicians. Nori\cite{Nor82} showed the K\"unneth formula for unipotent fundamental group scheme under any field; Mehta \& Subramanian\cite{MeSu02} showed the K\"unneth formula for Nori fundamental group scheme under algebraically closed field; Amrutiya \& Biswas\cite{AmBi10} showed the K\"unneth formula for F-fundamental group scheme under algebraically closed field; Langer\cite{Lan12} showed the K\"unneth formula for S-fundamental group scheme under algebraically closed field; Zhang\cite{Zha13} showed the K\"unneth formula for local fundamental group scheme under perfect field; Zhang\cite{Zha13} showed the homotopy sequence for \'etale fundamental group scheme under perfect field is exact, which implies the K\"unneth formula for \'etale fundamental group scheme under perfect field; Otabe\cite{Ota17} showed the K\"unneth formula for EN-fundamental group scheme under field of characteristic 0. We use new method to generalize the K\"unneth formula for certain fundamental group schemes to more general case.

\begin{Proposition}\label{kunnethall}
Let $k$ be a field, $X$ and $Y$ geometrically reduced connected schemes proper over $k$, $x\in X(k)$, $y\in Y(k)$, $p_X: X\times_k Y\rightarrow X$ and $p_Y: X\times_k Y\rightarrow Y$ the projections. Then the natural homomorphism $\pi^{*}(X\times_k Y,(x,y))\rightarrow \pi^{*}(X,x)\times_k \pi^{*}(Y,y)$ is an isomorphism for $*\in\{S,N,F,Loc,\acute{e}t,uni,EN,EF,ELoc,E\acute{e}t\}$.
\end{Proposition}

\begin{proof}
Consider the following Cartesian diagram:
\[
    \begin{tikzcd}
        X_{\bar{k}}\times_{\bar{k}}Y_{\bar{k}}\arrow[dr, phantom, "\ulcorner", very near start]\arrow[r,"p_{X_{\bar{k}}}"]\arrow[d,"p_2"]&X_{\bar{k}}\arrow[dr, phantom, "\ulcorner", very near start]\arrow[r]\arrow[d,"p_1"]&\Spec {\bar{k}}\arrow[d]\\
        X\times_k Y\arrow[r,"p_X"]&X\arrow[r]& \Spec k.
    \end{tikzcd}
\]
For any $E\in\mathcal{C}^{NF}(X\times_k Y)$, by Lemma~\ref{NFiff} we have $E\otimes_k \bar{k}\in \mathcal{C}^{NF}(X_{\bar{k}}\times_{\bar{k}}Y_{\bar{k}})$. By Lemma~\ref{Slangerkunneth} and Theorem~\ref{kunnethMain}, we have ${p_{X_{\bar{k}}}}_*(E\otimes_k \bar{k})\in\mathcal{C}^{NF}(X_{\bar{k}})$ and $E\otimes_k \bar{k}$ satisfies base change at $x_{\bar{k}}\in X_{\bar{k}}(\bar{k})$ lying over $x\in X(k)$. 

Consider cohomology and flat base change 
$$({p_X}_*E)\otimes_k \bar{k}=p_1^*{p_X}_*E\cong {p_{X_{\bar{k}}}}_*p_{2}^*E={p_{X_{\bar{k}}}}_*(E\otimes_k \bar{k})\in\mathcal{C}^{NF}(X_{\bar{k}}).$$
Then ${p_X}_*E\in\mathcal{C}^{NF}(X)$ by Lemma~\ref{NFiff} and $(({p_X}_*E)\otimes_k \bar{k})|_{x_{\bar{k}}}\cong ({p_{X_{\bar{k}}}}_*(E\otimes_k \bar{k}))|_{x_{\bar{k}}}$ by taking fibre at $x_{\bar{k}}$.

Let $i_Y:Y\rightarrow X\times_k Y$ send $y'$ to $(x,y')$ and $i_{Y_{\bar{k}}}:Y_{\bar{k}}\rightarrow X_{\bar{k}}\times_{\bar{k}}Y_{\bar{k}}$ send $y_{\bar{k}}'$ to $(x_{\bar{k}},y'_{\bar{k}})$ where $x_{\bar{k}}\in X_{\bar{k}}(\bar{k})$ lies over $x\in X(k)$. Consider the following commutative diagram
\[
    \begin{tikzcd}
        &Y\arrow[rr,"i_Y" near start]\arrow[dd]&&X\times_k Y\arrow[dd,"p_X" near start]\\
        Y_{\bar{k}}\arrow[ru,"p_3"]\arrow[rr,crossing over,"i_{Y_{\bar{k}}}" near start]\arrow[dd]&&X_{\bar{k}}\times_{\bar{k}}Y_{\bar{k}}\arrow[ru,"p_2"]&\\
        &x\arrow[rr]&&X\\
        x_{\bar{k}}\arrow[rr]\arrow[ru,"p_4"]&&X_{\bar{k}}\arrow[from=uu,crossing over,"p_{X_{\bar{k}}}" near start]\arrow[ru,"p_1"]&
    \end{tikzcd}.
\]
Then for any $E\in\mathcal{C}^{NF}(X\times_k Y)$, we have a natural homomorphism
$$\varphi:({p_X}_*E)|_{x}\rightarrow H^0(Y,i_Y^*E),$$
which yields a natural homomorphism
$$\varphi_{\bar{k}}:({p_X}_*E)|_{x}\otimes_k \bar{k}\rightarrow H^0(Y,i_Y^*E)\otimes_k \bar{k}.$$
Since $E\otimes_k \bar{k}$ satisfies base change at $x_{\bar{k}}$, we have the following commutative diagram
$$\begin{tikzcd}
    ({p_X}_*E)|_{x}\otimes_k \bar{k}\arrow[d,"\varphi_{\bar{k}}"]\arrow[r,equal]& (({p_X}_*E)\otimes_k \bar{k})|_{x_{\bar{k}}}\arrow[r,"\cong"]&({p_{X_{\bar{k}}}}_*(E\otimes_k \bar{k}))|_{x_{\bar{k}}}\arrow[d,"\cong"]\\
    H^0(Y,i_Y^*E)\otimes_k \bar{k}\arrow[r,"\cong"]& H^0(Y_{\bar{k}},(i_Y^*E)\otimes_k \bar{k})\arrow[r,equal]& H^0(Y_{\bar{k}},i_{Y_{\bar{k}}}^*(E\otimes_k \bar{k}))
\end{tikzcd}
$$
It follows that $\varphi:({p_X}_*E)|_{x}\xrightarrow{\cong} H^0(Y,i_Y^*E)$, i.e. $E$ satisfies base change at $x$. Hence the natural homomorphism $\pi^{S}(X\times_k Y,(x,y))\rightarrow \pi^{S}(X,x)\times_k \pi^{S}(Y,y)$ is an isomorphism.

For the other case, it follows by the natural isomorphism $\pi^{S}(X\times_k Y,(x,y))\rightarrow \pi^{S}(X,x)\times_k \pi^{S}(Y,y)$, Proposition~\ref{kunnethdescent}, Remark~\ref{tannakianrelation}, Proposition~\ref{pullbackall} and Proposition~\ref{pullbacksaturated}.
\end{proof}

\begin{Corollary}\label{pushforwardproduct}
Let $k$ be a field, $X$ and $Y$ geometrically reduced connected schemes proper over $k$, $p_X: X\times_k Y\rightarrow X$ the projection, $*\in\{S,N,F,Loc,\acute{e}t,uni\}$. Then for any $E\in\mathcal{C}^{*}(X\times_k Y)\text{ }(\text{resp. }\overline{\mathcal{C}^{*}(X\times_k Y)})$, we have ${p_X}_*E\in\mathcal{C}^{*}(X)\text{ }(\text{resp. }\overline{\mathcal{C}^{*}(X)})$.
\end{Corollary}

\begin{proof}
    It follows by Theorem~\ref{kunnethMain} and Proposition~\ref{kunnethall}.
\end{proof}

\section*{Acknowledgements}
We would like to thank Adrian Langer, Junchao Shentu and Lei Zhang for their enlightening comments and helpful conversations. The authors are partially supported by Applied Basic Research Programs of Science and Technology Commission Foundation of Shanghai Municipality(22JC1402700), National Natural Science Foundation of China(Grant No. 12171352).

\end{document}